\documentclass[12pt,authoryear,letterpaper,reqno,oneside]{amsart}
\usepackage[left=1.2in,right=1.2in,bottom=1.2in,top=1.2in]{geometry}
\usepackage{amsmath,amsfonts,amssymb}
\usepackage{comment}
\usepackage[all]{xy}
\usepackage[usenames]{color}
\usepackage{setspace}
\usepackage{mathrsfs}
\usepackage{stmaryrd}
\usepackage[colorlinks,citecolor=darkblue,urlcolor=darkblue,linkcolor=darkgreen]{hyperref}
\DeclareMathOperator{\Aut}{Aut}
\DeclareMathOperator{\Orb}{Orb}
\DeclareMathOperator{\dcl}{dcl}
\DeclareMathOperator{\Models}{Str}

\theoremstyle{plain}%
\newtheorem{theorem}{Theorem}[section]%

\newtheorem{proposition}[theorem]{Proposition}%
\newtheorem{lemma}[theorem]{Lemma}%
\newtheorem{corollary}[theorem]{Corollary}%
\newtheorem{definition}[theorem]{Definition}%

\newcommand{\HF}{\mathrm{HF}}

\newcommand{\Fr}{Fra{\"i}ss{\'e}}
\newcommand{\LMS}{{L_\mathrm{MS}}}
\newcommand{\TMS}{{T_\mathrm{MS}}}

\newcommand{\Ltr}{{L_\mathrm{tr}}}

\def\st{\,:\,}

\def\M{{\EM{\mathcal{M}}}}
\def\SS{{\EM{\mathcal{S}}}}
\def\W{{\EM{\mathcal{P}}}}
\def\cU{{\EM{\mathcal{U}}}}
\def\cM{{\EM{\mathcal{M}}}}
\def\N{{\EM{\mathcal{N}}}}
\def\cN{{\EM{\mathcal{N}}}}
\def\X{{\EM{\mathcal{X}}}}

\def\fT{{\EM{\Sigma}}}

\def\AA{{\EM{\mathcal{A}}}}
\def\BB{{\EM{\mathcal{B}}}}

\def\Pr{{\EM{\mathbb{P}}}}

\def\cK{{\EM{\mathcal{K}}}}

\def\CC{{\EM{\mathcal{C}}}}

\newcommand{\llrr}[1]{{\llbracket #1 \rrbracket}}
\newcommand{\llrrC}[1]{{\llrr{#1}_{\cC_0}}}
\newcommand{\llrrAG}[1]{{\llrr{#1}_{\AG}}}
\newcommand{\llrrsub}[2]{{\llbracket #1 \rrbracket}_{#2}}
\newcommand{\bigllrr}[1]{{\bigl \llbracket #1 \bigr \rrbracket}}
\newcommand{\bigllrrAG}[1]{{\bigl \llbracket #1 \bigr \rrbracket}_{\AG}}

\def\cC{{\EM{\mathscr{C}}}}

\def\O{{\EM{\mathcal{O}}}}
\def\UU{{\EM{\mathbb{U}}}}

\def\x{{\EM{\ol{x}}}}
\def\xx{{\EM{\ol{x}}}}
\newcommand\dd{{\EM{\mathbf{d}}}} 
\def\y{{\EM{\ol{y}}}}

\def\z{{\EM{\ol{z}}}}
\def\s{{\EM{\ol{s}}}}
\def\zz{{\EM{\ol{z}}}}
\def\ttt{{\EM{\ol{t}}}}
\def\a{{\EM{\ol{a}}}}

\def\aa{{\EM{\ol{a}}}}
\def\vv{{\EM{\ol{v}}}}
\def\ww{{\EM{\ol{w}}}}
\def\uu{{\EM{\ol{u}}}}
\def\b{{\EM{\ol{b}}}}

\def\c{{\EM{\ol{c}}}}

\def\m{{\EM{\ol{m}}}}

\newcommand{\wlw}{\ensuremath{\Nats^{<\w}}}
\newcommand{\wtow}{\ensuremath{\Nats^{\w}}}
\def\ov#1{{\EM{\overline{#1}}}}
\newcommand{\qn}{q^\natural}

\def\Lomega#1{{\EM{\mc{L}_{#1, \w}}}}
\def\Lww{\Lomega{\w}}
\def\Lwow{\Lomega{\w_1}}

\newcommand{\defn}[1]{{\bf{#1}}}
\newcommand{\defas}{{\EM{\ :=\ }}}
\newcommand{\MLZ}{\Models_{L_0, L}^{\M_0}}
\newcommand{\MLG}{\Models_{L_G, L}^{\AG}}
\newcommand{\AG}{{\AA_G}}

\newcommand{\ER}{Erd\H{o}s-R\'enyi}

\def\w{\EM{\omega}}

\def\Reals{{\EM{{\mbb{R}}}}}
\def\Rationals{{\EM{{\mbb{Q}}}}}

\def\RatNonneg{{\EM{{\mbb{Q}_{\ge 0}}}}}
\def\Rplus{{\EM{{\mbb{R_+}}}}}

\def\Naturals{{\EM{{\mbb{N}}}}}
\def\Nats{{\EM{{\mbb{N}}}}}

\def\^{\EM{{}^{\And}}}

\def\And{\EM{\wedge}}

\def\<{\EM{\langle}}
\def\>{\EM{\rangle}}

\def\EM#1{\ensuremath{#1}}
\def\mbb#1{\EM{\mathbb{#1}}}

\def\mc#1{\EM{\mathcal{#1}}}

\def\ol#1{\EM{\overline{#1}}}

\def\ul#1{\underline{#1}}

\newcommand{\logicact}{\circledast_L}
\newcommand{\logicactcz}{\circledast_{\cC_0, L}}
\newcommand{\sympar}[1]{\ensuremath{S_{#1}}}

\newcommand{\sym}{\sympar{\infty}}
\newcommand{\vinf}{\ensuremath{V_\infty}}
\newcommand{\vfin}{\ensuremath{V_{\mathrm{fin}}}}

\newcommand{\Ginf}{\ensuremath{G_\infty}}
\newcommand{\Gfin}{\ensuremath{G_{\mathrm{fin}}}}
\newcommand{\mufin}{\ensuremath{\mu_{\mathrm{fin}}}}

\newcommand{\GG}{\ensuremath{\mathbb{G}}}

\newcommand{\Full}{\ensuremath{\mathrm{Full}}}

\newcommand{\tind}{\ensuremath{t_{\mathrm{full}}}}

\definecolor{darkred}{rgb}{0.5,0,0}
\definecolor{darkgreen}{rgb}{0, 0.3,0}
\definecolor{darkblue}{rgb}{0,0,0.6}

\begin{document}


\title[Invariant measures via inverse limits of finite structures]
{Invariant measures via inverse limits \\ of finite structures}

\author[Ackerman]{Nathanael Ackerman}
\address{
Department of Mathematics\\
Harvard University\\
One Oxford Street\\
Cambridge, MA 02138\\
USA
}
\email{nate@math.harvard.edu}

\author[Freer]{Cameron Freer}
\address{
Department of Brain and Cognitive Sciences\\
Massachusetts Institute of Technology\\
77 Massachusetts Ave.\\
Cambridge, MA 02139\\
USA
}
\email{freer@mit.edu}

\author[Ne{\v{s}}et{\v{r}}il]{Jaroslav Ne{\v{s}}et{\v{r}}il}
\address{Department of Applied Mathematics and Institute of Theoretical
Computer Science (IUUK and ITI)\\ Charles University\\ Malostransk\'e
n\'am.25\\ 11800
Praha 1\\ Czech Republic}
\email{nesetril@kam.ms.mff.cuni.cz}

\author[Patel]{Rehana Patel}
\address{Franklin W.\ Olin College of Engineering\\ Olin Way\\ Needham,
MA\\ 02492\\ USA}
\email{rehana.patel@olin.edu}

\begin{abstract}
Building on recent results 
regarding
symmetric probabilistic constructions of
countable structures, we
provide a method for constructing probability measures,
concentrated on certain classes of countably infinite structures,
that are invariant under 
all permutations
of the
underlying set that fix all 
constants.
These measures are constructed 
from
inverse limits of measures on
certain finite structures.
We use this construction to
obtain 
invariant probability measures concentrated on the classes
of
countable models of
certain
first-order theories, including 
measures
that do not assign positive
measure to the isomorphism class of any single model.
We also characterize those transitive
Borel $G$-spaces admitting a $G$-invariant probability measure, when $G$ is
an arbitrary countable product of symmetric groups on a countable
set.
\end{abstract}


\maketitle
\thispagestyle{empty}

\vspace*{-20pt}

\begin{small}
\renewcommand\contentsname{\!\!\!\!}
\setcounter{tocdepth}{1}
\tableofcontents
\end{small}

\vspace*{-20pt}




\section{Introduction}

Symmetric probabilistic constructions of mathematical
structures have a long history,
dating back to the countable random graph model of
\ER~\cite{MR0120167}, a construction that
with  probability $1$
yields
(up to isomorphism) the Rado graph, i.e., the countable universal
ultrahomogeneous graph.
In this paper, we build on recent developments that have extended the range
of such constructions.  In particular, we consider when a symmetric probabilistic
construction can 
produce
many different countable structures, 
with no isomorphism
class occurring with positive probability. We also consider probabilistic
constructions with respect to various 
notions of partial symmetry.

One 
natural notion
of
a symmetric probabilistic construction
is via an \emph{invariant measure} --- namely, a probability measure on a class of
countably infinite structures that is invariant under all permutations of
the underlying set of elements. When such an invariant measure assigns
probability $1$ to a given class of structures (as the \ER\ construction does 
to
the isomorphism class
of the Rado graph), we say that it is \emph{concentrated} on such
structures, and that the given class \emph{admits} an invariant measure.

For several decades, most known examples of such invariant measures were variants of
the \ER\ random graph,
for instance, an analogous construction 
that produces
the 
countable universal bipartite
graph.
In recent years, a number of other important classes of
structures
have been shown to
admit invariant measures, most notably 
the collection of countable metric spaces whose completion is Urysohn space, by Vershik
\cite{MR2006015}, \cite{MR2086637}, and Henson's universal ultrahomogeneous
$K_n$-free graphs by Petrov
and Vershik
\cite{MR2724668}. Both constructions
are considerably more complicated than
the \ER\ construction.
By extending the methods of \cite{MR2724668}, Ackerman, Freer, and Patel
\cite{AFP}
have completely 
characterized
those countable structures in a countable
language whose isomorphism class admits 
an
invariant 
measure.

In the present paper
we extend 
the 
construction of \cite{AFP}.
Our new construction is more 
streamlined than 
the one
in \cite{AFP}, and also 
broader
in its consequences. 
Both constructions involve building
continuum-sized structures from which 
invariant measures are obtained by sampling, but 
the one in \cite{AFP} produces
an explicit structure with underlying set the real numbers, necessitating
various book-keeping devices, which we avoid here.

As a first
application of 
the present
more general construction, we 
describe certain
first-order
theories 
having
the property that
there is an invariant probability measure that is concentrated on the class of models
of the theory but that assigns measure $0$ to the
isomorphism class of 
each
particular
model.
We thereby obtain new examples of 
classes of
structures 
admitting invariant
measures,
and new examples of invariant measures concentrated on 
collections of structures 
that were 
previously
known to admit invariant measures.

Towards
our second application, we consider measures that are invariant
under 
the action of certain
subgroups
of the full permutation group $\sym$ on the underlying set.
Note
that any
 random construction of a countably infinite 
structure 
with
constants
faces a fundamental obstacle
to having an $\sym$-invariant distribution, as described in \cite{AFP}.
Namely,
if the distribution were $\sym$-invariant,
then
the probability 
that
any given constant symbol in the language 
is
interpreted as a
particular element 
would have to
be 
the same as for any other
element, leading to a contradiction, as a countably infinite set of
identical reals cannot sum to $1$.
In other words, if 
a structure admits 
an $\sym$-invariant
measure,
then it cannot be in a language having constant symbols.
Furthermore, 
if a measure
concentrated on the isomorphism class of the structure
is invariant 
under 
a given
permutation, then that permutation
must
fix all elements that interpret constant symbols.

With that obstacle in mind,
we may ask, more generally, which structures
admit
measures that are invariant under all permutations of the underlying
set of the structure and that fix the restriction of the structure to  a particular
sublanguage.
We answer this question 
in the case of a unary sublanguage, 
i.e., where the sublanguage consists entirely of unary relations.
By results in descriptive set theory, 
this
is
equivalent to describing all those transitive Borel $G$-spaces admitting a
$G$-invariant probability measure when $G$ is a
countable product of symmetric groups on a countable (finite or
infinite) set.
This constitutes the second application of our construction.

In the special case 
of
undirected graphs, our
methods 
for
producing invariant measures can be viewed as constructing
dense graph limits, in the sense of
Lov\'asz and Szegedy \cite{MR2274085} and others; for details, see \cite{MR3012035}.
In fact, by results of
Aldous \cite{MR637937}, Hoover \cite{Hoover79}, Kallenberg \cite{MR1182678},
and Vershik \cite{MR1922015}  in work on the probability theory of
exchangeable arrays, 
an
invariant measure on graphs is
necessarily 
the distribution
of a particular sampling procedure from \emph{some} continuum-sized limit
structure.
For more details on this connection, see Diaconis and Janson
\cite{MR2463439} and
Austin \cite{MR2426176}.

Our work also has
connections to
a recent study of Borel models of
size continuum
by Baldwin, Laskowski, and Shelah \cite{BaldwinLaskowski}, building on
work of Shelah \cite[Theorem~VII.3.7]{MR1083551}. Their continuum-sized
structures, like ours, are
constructed from
inverse limits;
however, our methods differ from theirs
in 
several respects
and, unlike 
\cite{BaldwinLaskowski}, 
our focus is on the
consequences of these constructions for invariant measures.

\subsection{Outline of the paper}

In 
Section~\ref{preliminaries}, we provide preliminaries 
for our constructions,
including definitions and basic results from the model theory of infinitary
logic and from descriptive set theory.

We then pause, in Section~\ref{simplified}, to provide a 
toy
construction, for graphs, that will
motivate 
the
more technical aspects
of our main construction.

In Section~\ref{invlimitconstruction},
we present our main technical construction,
in which we build
a 
special kind of
continuum-sized structure from
inverse
limits.

In the 
following
sections, we
provide two applications of 
this
main 
construction. First, in
Section~\ref{almostSec}, we use it to provide new constructions
of invariant probability measures concentrated
on the class of models of certain first-order theories, but assigning positive measure
to no single isomorphism class.

Second, in Section~\ref{GorbitSec},
we use the 
main construction
to characterize those structures that are invariant under automorphism
groups that fix the restrictions of the structures to unary sublanguages.
As noted, this
amounts to characterizing 
those transitive Borel $G$-spaces
that admit
a $G$-invariant probability measure, when $G$ is a countable
product of symmetric groups on a countable (finite or infinite) set.

\section{Preliminaries}
\label{preliminaries}

In this section, we describe
some 
notation, and
introduce 
several basic
notions
regarding
infinitary logic,
transitive $G$-spaces, and model-theoretic structures and their automorphisms
that we
will use throughout the paper.

The set $\wlw$ is defined to be the collection of finite sequences of
natural numbers.  For $x, y \in \wlw$ we write $x \preceq y$ when $x$ is an
initial segment of $y$. The set $\wtow$ is the collection of countably
infinite sequences of natural numbers. For $x \in \wtow$, we write $x|_n$
to denote the length-$n$ initial segment of $x$ in $\Nats^n$, and similarly
for elements of $\wlw$ of length at least $n$.

Suppose $j\in\Nats$. For $x_0, \ldots, x_j, y_0, \ldots, y_j \in \wlw$, we write
\[(x_0, \ldots, x_j)\sqsubseteq (y_0, \ldots, y_j)\]
when $x_i \preceq y_i$
for $0 \le i \le j$.

We write $a\^ b$ to denote the concatenation of $a, b\in\wlw$, though we
often omit the symbol $\^$ when concatenating explicit sequences.
Occasionally we will use exponential notation for repeated numerals; e.g.,
$0^42^2$ denotes $000022\in\wlw$.
Define the \defn{projection function}
\[\pi\colon \wlw \to \wlw\]
by
\[
\pi(a\^ b) = a
\]
when $a\in\wlw$ and $b \in \Naturals$, and
\[
\pi(\<\,\>) = \<\,\>,
\]
where $\<\,\>$ denotes the empty string.
Write the composition of projection with itself as $\pi^2 \defas \pi \circ \pi$. We will use this notation in
\S\ref{construction-subsec}.

Define  $\Rplus \defas \{ x \in \Reals \st x >0\}$
and
$\RatNonneg \defas \{ x \in \Rationals \st x \ge 0\}$.

A probability measure on $\Reals$ is said to be \defn{non-degenerate} when every
non-empty open set has positive measure and \defn{atomless} when every singleton has measure $0$.

We say that a probability measure $\mu$ on an arbitrary measure 
space $S$ is \defn{concentrated} on a measurable set $X\subseteq S$ when
$\mu(X) = 1$.
Given a measurable action of a group $G$ on $S$, we say that $\mu$ is \defn{$G$-invariant} if
$\mu(X) = \mu(g\cdot X)$ for every $g\in G$ and measurable $X\subseteq S$.

\subsection{Model theory of infinitary logic}

We now briefly recall notation for finitary and infinitary formulas. For
more details on such formulas and on the corresponding notion of
satisfiability (denoted by $\models$), see \cite{MR0424560} 
and \cite[\S1.1]{MR1924282}.
Throughout this paper, 
$L$  will
be a countable language, i.e., a countable collection of relation,
constant, and function symbols. Fix
an implicit set of countably
infinitely many variables. Then $\Lww(L)$ is the set of all
(finitary) first-order formulas (in that set of variables) with relation, constant, and function
symbols from $L$. The set $\Lwow(L)$ of infinitary $L$-formulas is the
smallest set containing $\Lww(L)$ and closed under countable conjunctions,
existential quantification, and negation, and such that each formula has
only finitely many free variables. In particular, $\Lwow(L)$ is closed
under taking subformulas. A sentence is a formula having no free variables,
and a theory is an arbitrary collection of sentences.

Let $k \in \Nats$ and let $x_1, \ldots, x_k$ be
distinct variables. A (complete) \defn{quantifier-free
\linebreak $L$-type} $q$ with free variables $x_1, \ldots, x_k$ is a
countable collection of quantifier-free formulas of $\Lwow(L)$ whose set of free variables
is contained in $\{x_1, \ldots, x_k\}$,
and such that for any
quantifier-free
$\Lwow(L)$-formula
$\psi$
whose free variables are among $x_1, \ldots, x_k$,
either
\[\models (\forall x_1, \ldots, x_k) \bigl(\bigwedge_{\varphi \in q}
\varphi \to \psi \bigr)
\qquad \text{or}
\qquad
\models (\forall x_1, \ldots, x_k) \bigl(\bigwedge_{\varphi \in q} \varphi
\to \neg \psi \bigr).
\]
Note that any collection $q$ of formulas which has this property with
respect to all 
\emph{atomic} formulas $\psi\in \Lww(L)$ is already
a complete quantifier-free $L$-type.

Note that we will consider quantifier-free types to 
entail
a fixed ordering of their free variables.
This will be important because for a quantifier-free type $q$ with $k$-many
free variables, and a set $X$ of size $k$
with a specified ordering $<$,  we will sometimes write $q(X)$ to represent
the statement that $q(\ell_1, \ldots, \ell_{k})$ holds, where $\ell_1 < \cdots <
\ell_{k}$ are the elements of $X$.

We say that a quantifier-free type with free variables $x_1, \ldots, x_k$
is \defn{non-constant} when it implies that none of $x_1, \ldots, x_k$
instantiates a constant symbol,  and is
\defn{non-redundant} when it implies
\[
\bigwedge_{1\le i < j \le k} (x_i \neq x_j).
\]

Suppose $L_0$ is a sublanguage of $L$, i.e., each of the sets of relation,
constant, and function symbols of $L_0$ is a subset of the corresponding
set for $L$. Then the \defn{restriction} $q |_{L_0}$ of a quantifier-free
$L$-type to $L_0$ is defined to be set of atomic $L_0$-formulas  and their
negations that are implied by $\bigwedge_{\varphi \in q} \varphi$.

An
$L$-theory $T$
is \defn{quantifier-free complete}
when it is consistent and for every quantifier-free $L$-sentence $\varphi$, exactly one of $T \models
\varphi$ or $T \models \neg \varphi$ holds. 

We will later make use of the notion of a \defn{Scott sentence}: a sentence of $\Lwow(L)$
which characterizes a given countable structure up to isomorphism among
other countable $L$-structures. For more details, see
\cite[Corollary~VII.6.9]{MR0424560}. We will also 
use the notion of
an admissible set, and in particular the admissible set $\HF$ of hereditarily finite sets; again see \cite{MR0424560}.

For a structure $\M$ with underlying set $M$, a natural number
$k\in \Nats$,
and a $k$-tuple $\aa = (a_1, \ldots, a_k) \in M^k$,
we will sometimes abuse notation and write either $\aa\in M$ or $\aa \in
\M$ to mean that $a_1,
\ldots, a_k \in M$. We will also sometimes write $a_1 \cdots a_k$ to denote such a tuple.

Suppose $\M$ is an 
$L$-structure.  
When $U$ is a relation symbol in $L$, we
write
$U^\M$ to denote the set of tuples $\aa \in \M$ such that  $\M \models
U(\aa)$. Similarly, we write $c^\M$ for the instantiation in $\M$ of a
constant symbol $c\in L$ and $f^\M$ to denote the function on $\M$-tuples corresponding
to the function symbol $f \in L$.
Given a sublanguage $L_0 \subseteq L$, we write $\cM|_{L_0}$ to denote the restriction of $\M$ to $L_0$.

\subsection{Definitional expansions}
\label{defExpSec}

Fundamental to our main construction is a special sort of sentence.
We define the \defn{pithy $\Pi_2$ sentences} of $\Lwow(L)$ to be those
\linebreak $\Lwow(L)$-sentences that are of the form
\[(\forall \x)(\exists y)\varphi(\x, y),\]
where $\varphi \in \Lwow(L)$ is quantifier-free with free variables
precisely $\x, y$, and where the tuple $\x$ of variables is possibly empty.
We say that a theory $T \subseteq \Lwow(L)$ is pithy $\Pi_2$ when each
sentence in $T$ is.

In Sections~\ref{almostSec} and \ref{GorbitSec} we
will make use of the following technical result,
which produces a 
definitional expansion of the empty theory
to a pithy $\Pi_2$ theory $\fT_A$ in which every formula in a desired admissible 
set $A$
is
equivalent to a quantifier-free formula;
we call $\fT_A$ the \defn{definitional expansion for $A$}.
This result is a straightforward extension of the standard
\emph{Morleyization} method.

\begin{lemma}
\label{morleyization}
For every
admissible set $A \supseteq L$, there is
an expanded language $L_A \subseteq A$  and a
pithy $\Pi_2$ theory $\fT_A \subseteq \Lwow(L_A) \cap A$
such that 
\begin{itemize}
\item[(i)]
for every formula $\varphi \in \Lwow(L)\cap A$,
there is some atomic formula $R_\varphi \in L_A$ such that
\[\fT_A \models (\forall \xx)\Bigl[\bigl((\forall w) R_\varphi(\xx,w) \leftrightarrow \varphi(\xx) \bigr) \And \bigl((\exists w)R_\varphi(\xx,w) \leftrightarrow \varphi(\xx) \bigr)\Bigr], \] 
where $\xx$ is the tuple of free variables of $\varphi$,
\item[(ii)]
every $L$-structure has a
unique expansion to an $L_A$-structure that satisfies $\fT_A$, and
\item[(iii)]
$\fT_A$ implies that every atomic formula of $\Lww(L_A) \setminus \Lww(L)$
is equivalent to some formula of $\Lwow(L)\cap A$.
\end{itemize}
\end{lemma}
\begin{proof}
Consider the countable language $L_A \defas L \cup \{R_\psi\st \psi\in A\} $, where each relation symbol $R_\psi$ is a distinct element of $A\setminus L$ and has arity one more than the number of free variables in $\psi$.  

Let $\fT_A$ be the countable $\Lwow(L_A)$-theory consisting of the following $\Pi_2$ sentences:
\begin{itemize}
\item 
$(\forall \xx,w)[R_P(\xx,w) \leftrightarrow P(\xx)]$ for $P$ a relation symbol in $L$ of arity $|\xx|$,
\item 
$(\forall \xx,w)[R_{c}(y, w) \leftrightarrow
 c=y]$ for $c$ a constant symbol in $L$,
\item 
$(\forall \xx,w)[R_{f}(\xx, y, w) \leftrightarrow
 f(\xx)=y]$ for $f$ a function symbol in $L$ of arity $|\xx|$,
\item $(\forall \xx,w)[R_{\neg \psi}(\xx,w) \leftrightarrow \neg R_\psi(\xx,w)]$,
\item $(\forall \xx,w)[R_{\bigwedge_{i\in I} \psi_i}(\xx,w) \leftrightarrow \bigwedge_{i\in I}R_{\psi_i}(\zz_i,w)]$,
\item $(\forall \xx,w)[R_{(\exists y)\varphi}(\xx,w) \leftrightarrow (\exists y) R_{\varphi}(\xx,y,w)]$, and
\item $(\forall \xx,w)[R_{(\exists y)\psi}(\xx,w) \leftrightarrow (\exists y) R_{\psi}(\xx,w)]$,
\end{itemize}
where $\xx$ is a tuple containing precisely the free variables of $\psi \in A$, where $\bigwedge_{i\in I} \psi_i\in A$, where  the tuple $\zz_i\subseteq \xx$ contains precisely the free variables of $\psi_i$  for each $i \in I$, and where the free variables  of $\varphi\in A$ are precisely the variables in $\xx y$, with $y\not\in\xx$.

Note that $(\forall \xx)[\varphi(\xx) \leftrightarrow \psi(\xx)]$ is equivalent to 
$(\forall \xx)[\varphi(\xx) \to \psi(\xx)] \wedge (\forall \xx)[\psi(\xx) \to \varphi(\xx)]$.
Hence $\fT_A$ is equivalent to a theory all of whose axioms are either $\Pi_1$ or pithy $\Pi_2$.
Further, every $\Pi_1$ sentence is equivalent to some pithy $\Pi_2$ sentence. Hence we may assume without loss of generality that 
$\fT_A$ itself is a pithy $\Pi_2$ theory.

Observe that $\fT_A \subseteq A$ and that
\[\fT_A \models (\forall \xx)\Bigl[\bigl((\forall w) R_\varphi(\xx,w) \leftrightarrow \varphi(\xx) \bigr) \And \bigl((\exists w)R_\varphi(\xx,w) \leftrightarrow \varphi(\xx) \bigr)\Bigr], \] 
for all $\varphi\in \Lwow(L_A)\cap A$, where $\xx$ is the tuple of free variables of $\varphi$.

Note that in the definition of $\fT_A$, we included the dummy variable $w$ in order to ensure that for every $\psi\in A$, there is a universal formula that is equivalent to $\psi$ in every model of $\fT_A$, even for quantifier-free $\psi$.  
This is needed in order for
$\fT_A$ to itself be pithy $\Pi_2$, 
which often is not required
in the usual first-order Morleyization procedure \cite[Theorem~2.6.6]{MR1221741}.

An immediate generalization of
\cite[Theorem~2.6.5]{MR1221741} to countable fragments of $\Lwow(L)$ shows that
every $L$-structure $\cM$ has a unique expansion to an $L_A$-structure that satisfies $\fT_A$.
Finally, $\fT_A$ implies that every atomic formula of $\Lww(L_A) \setminus \Lww(L)$ is
equivalent to some formula of $\Lwow(L)\cap A$.
\end{proof}

We will make use of Lemma~\ref{morleyization} in the proof of Proposition~\ref{backward}.

For a first-order theory $T \subseteq \Lww(L)$, we define the \defn{pithy $\Pi_2$ expansion of $T$} to be the $\Lww(L_\HF)$-theory 
\[
\fT_\HF \cup \{ (\forall x) R_\varphi(x) \st \varphi \in T\},
\]
where $\HF$ denotes the hereditarily finite sets. 
We will make use of this notion in
Lemma~\ref{Tstar-pithy} and Theorem~\ref{SplittingTypesTheorem}.

\subsection{\Fr\ limits and trivial definable closure}
\label{Frlimits}
Suppose that the countable language $L$ is \defn{relational}, i.e., does not
contain constant or function symbols.
The \defn{age} of an $L$-structure  $\M$ is defined to be the class of all
finite $L$-structures isomorphic to 
a
substructure of $\M$.

A countable $L$-structure $\M$ is said to be
\defn{ultrahomogeneous} when
 any partial isomorphism between finite substructures of $\M$ can be extended to automorphism of $\M$.
Any two ultrahomogeneous countably infinite $L$-structures have the same
age if and only if they are isomorphic.
The age  of any ultrahomogeneous countably infinite \linebreak $L$-structure
is a 
class that contains
countably infinitely many isomorphism types and that satisfies
the so-called 
hereditary property, 
joint embedding property,
and
amalgamation property.
Conversely,
any 
class of finite $L$-structures that is closed under isomorphism, contains countably infinitely many isomorphism types, and that satisfies
these
three
properties is the age of some ultrahomogeneous countably infinite $L$-structure, in fact a unique such structure (up to isomorphism),
called its \defn{\Fr\ limit};
such a class of finite structures is called an \defn{amalgamation class}.
An amalgamation class 
is called a \defn{strong amalgamation class} when it
further satisfies  the strong amalgamation property
--- namely, when any two elements of the class can be amalgamated over any finite common 
substructure in a non-overlapping way.

It is a standard fact that the first-order theory of any \Fr\ limit in a finite relational
language has an axiomatization consisting of 
pithy
$\Pi_2$ sentences that are first-order. These axioms are often referred to as \emph{(one-point)
extension axioms}.
For more details, see, e.g., \cite[\S7.1]{MR1221741}.

Let $\M$ be an $L$-structure and let $M$ be its underlying set. Suppose 
$X \subseteq M$.
The \defn{definable closure} of $X$ in $\M$, written
$\dcl(X)$,
is the set of all elements of $M$ that are fixed by every
automorphism of $\M$ fixing $X$ pointwise. We say that $\M$ has
\defn{trivial definable closure} when
$\dcl(\aa) = \aa$ for all finite tuples $\aa\in \M$.
An
ultrahomogeneous countably infinite structure
$\M$ 
in a relational language
has trivial
definable closure if and only if its age has the strong amalgamation
property (again see \cite[\S7.1]{MR1221741}).

\subsection{Transitive $G$-spaces}
Let $(G, e, \cdot)$ be a Polish group.
We now recall the notion of a \emph{transitive Borel
$G$-space}.
\begin{definition}
A \defn{Borel $G$-space} $(X, \circ)$ consists of
a Borel space $X$ along with a Borel map $\circ\colon G \times X \rightarrow X$ such that
\begin{itemize}
\item $(g \cdot h) \circ x = g \circ (h \circ x)$
for every $g, h \in G$ and $x \in X$, and
\item $ e \circ x = x$ for every $x  \in X$.
\end{itemize}
For Borel $G$-spaces 
$(X, \circ_X)$ and  $(Y, \circ_Y)$,
a \defn{map} 
$\tau$ between 
$(X, \circ_X)$ and  $(Y, \circ_Y)$
is a
Borel map 
$\tau \colon X  \rightarrow Y$
for which
$\tau(g\circ_X x) = g \circ_Y \tau(x)$ for all
$g \in G$ and  $x \in X$.

A Borel $G$-space  $(X, \circ)$ is a \defn{universal} Borel $G$-space when every other Borel $G$-space maps injectively into it.
\end{definition}

\begin{definition}
A Borel $G$-space $(X, \circ)$ is \defn{transitive}
when for every $x, y \in X$ there is some $g\in G$ such that $g \circ x = y$,
i.e.,
the action $\circ$ has a single orbit. Equivalently,
there is no proper subspace $Y \subseteq X$ such that $(Y, \circ)$ is also a Borel $G$-space.
\end{definition}

Note that in particular, any orbit of a Borel $G$-space is itself a
transitive Borel $G$-space under the restricted action.

The main result of Section~\ref{GorbitSec} is a classification of \emph{transitive} Borel $G$-spaces for certain groups $G$.

\subsection{Structures and automorphisms}
We consider three types of countable structures: those
with underlying set $\Nats$,
those
with a fixed countable set of constants disjoint from $\Nats$, and
those
with underlying set $\Nats$ whose restriction to a sublanguage
is some fixed structure.

\subsubsection{The Borel space of countable structures}

We now define the Borel space $\Models_L$ and its associated \emph{logic
action}.
These notions will be used throughout the
paper, and especially in Sections~\ref{simplified}, \ref{almostSec},
and \ref{GorbitSec}.

\begin{definition}
\label{StrDef}
Let $L$ be a countable language.
Define $\Models_{L}$ to be the set of $L$-structures with underlying set
$\Nats$.
\end{definition}

\begin{definition}
Let $L$ be a countable language.
Then for every $\Lwow(L)$-formula $\varphi$,
define
\[\llrr{\varphi(\ell_1, \dots, \ell_j)} \defas
 \{\M \in \Models_{L} \st \M\models \varphi(\ell_1, \dots, \ell_j)\}\]
for all $\ell_1, \ldots, \ell_j \in \Nats$, where $j\in\Nats$ is the number of free
variables (possibly $0$) of $\varphi$.
\end{definition}

When $\Models_{L}$ is equipped with the $\sigma$-algebra consisting
of all such sets
$\llrr{\varphi(\ell_1, \dots, \ell_j)}$,
it becomes a standard Borel
space; for details, see \cite[\S2.5]{MR1425877}.
Note that when we say that a probability measure is concentrated on some class of models of an $L$-theory, we mean
that the measure is concentrated on the restriction of that class to $\Models_L$.

\begin{definition}
For a non-empty set $A$, we write $\sympar{A}$ to denote the symmetric group on
$A$.
For $n\in\Nats$,
we
write $\sympar{n}$
to denote
$\sympar{\{0, \ldots, n-1\}}$, and we will use $\sym$ to denote $S_\Nats$, the symmetric group on $\Naturals$.
\end{definition}

\begin{definition}[{\cite[\S2.5]{MR1425877}}]
Let $L$ be a countable language.
Define
the Borel \linebreak $\sym$-action
\[\logicact \colon \sym \times
\Models_{L}
\to \Models_{L}\]
to be such that
for all $g\in \sym$ and $\M \in \Models_{L}$,
\[
g \logicact \M \models \varphi(\ell_1, \ldots, \ell_j)
\]
if and only if
\[
\M \models \varphi\bigl(g^{-1}(\ell_1), \ldots, g^{-1}(\ell_j)\bigr)
\]
for
all $\Lwow(L)$-formulas $\varphi$
and
all $\ell_1, \ldots, \ell_j \in\Nats$, where $j$ is the number
of free variables of $\varphi$.
\end{definition}

\subsubsection{Countable structures with a fixed set of constants}

We now define the analogous notions for the situation where we
instantiate constants by elements other than ones from $\Naturals$. We will
need these notions in Section~\ref{invlimitconstruction}.

\begin{definition}
\label{StrDefcC}
Let $L$ be a countable language and let $C$ be the set of its constant
symbols (possibly empty).
Let $C_0$ be a countable set  (empty when $C$ is empty) that is disjoint from $\Nats$, and
suppose 
$\cC_0 \colon C \to C_0$
is a surjective function.
Then define $\Models_{\cC_0, L}$ to be the set of $L$-structures with underlying set
$\Nats \cup C_0$ in which
the instantiation of $c$ is $\cC_0(c)$,
for each constant symbol $c\in C$.
In particular, 
no
element of $\Naturals$ 
instantiates
any constant symbol of $L$.
\end{definition}

Note that when $L$ has no constant symbols, then $C = C_0 = \emptyset$ and $\cC_0$
is the empty function, and we have
$\Models_{\cC_0, L} = \Models_L$.

\begin{definition}
Let $L$ be a countable language with $C$ its set of constant symbols, and
let $C_0$ and $\cC_0$ be as in Definition~\ref{StrDefcC}.
Then for every $\Lwow(L)$-formula $\varphi$,
define
\[\llrrsub{\varphi(\ell_1, \dots, \ell_j)}{\cC_0}
\defas
 \{\M \in \Models_{\cC_0, L} \st \M\models \varphi(\ell_1, \dots, \ell_j)\}\]
for all $\ell_1, \ldots, \ell_j \in \Nats$, where $j\in\Nats$ is the number of free
variables (possibly $0$) of $\varphi$.
\end{definition}

When $\Models_{\cC_0, L}$ is equipped with the 
$\sigma$-algebra consisting
of all such sets
\linebreak
$\llrrsub{\varphi(\ell_1, \dots, \ell_j)}{\cC_0}$,
it likewise becomes a
standard Borel space.

\begin{definition}
Let $L$ be a countable language with $C$ its set of constant symbols, and
let $C_0$ and $\cC_0$ be as in Definition~\ref{StrDefcC}.
Define $\sym^{C_0} \subseteq \sympar{\Nats\cup C_0}$ to be the subgroup consisting of all
permutations of $\Nats \cup C_0$ fixing $C_0$ pointwise.
Define
the Borel $\sym^{C_0}$-action
\[\logicactcz \colon \sym^{C_0} \times
\Models_{\cC_0, L}
\to \Models_{\cC_0, L}\]
to be such that
for all $g\in \sym^{C_0}$ and $\M \in \Models_{\cC_0, L}$,
\[
g \logicactcz \M \models \varphi(\ell_1, \ldots, \ell_j)
\]
if and only if
\[
\M \models \varphi\bigl(g^{-1}(\ell_1), \ldots, g^{-1}(\ell_j)\bigr)
\]
for
all $\Lwow(L)$-formulas $\varphi$
and
all $\ell_1, \ldots, \ell_j \in\Nats$, where $j$ is the number
of free variables of $\varphi$.
\end{definition}

Note that any permutation of $\Nats$ extends uniquely to a permutation of
$\Nats \cup C_0$ that fixes $C_0$ pointwise, and every such permutation
of $\Nats \cup C_0$ restricts to a permutation of $\Nats$, and
hence $\sym \cong \sym^{C_0}$.

\subsubsection{Relativized notions via sublanguages}
\label{MLZ-sec}
Finally, we consider structures with underlying set $\Nats$ whose restriction to a sublanguage
is some fixed structure. We will make use of 
such structures
in
Section~\ref{GorbitSec}.

\begin{definition}
Let $L$ be a countable language and let $\M$  be an $L$-structure with
underlying set $\Nats$.
We write $\Aut(\M)$ to denote the \defn{automorphism group} of $\M$, i.e., the
subgroup of $\sym$ consisting of all permutations of $\Nats$ that preserve
every relation, constant, and function of $\M$.
\end{definition}

\begin{definition}
Let $L$ be a countable language and let $L_0$ be a sublanguage of $L$.
Let $\M_0$ be an $L_0$-structure on $\Nats$.
Define $\Models_{L_0,L}^{\M_0}$ to be the collection of 
those structures in $\Models_L$
whose restriction to $L_0$ is
$\M_0$, i.e.,
\[
\Models_{L_0,L}^{\M_0}
\defas \{\M \in \Models_L \st \M|_{L_0} = \M_0
\}.
\]
\end{definition}

Note that when $L$ has no constant symbols, $L_0$ is the empty language,
$\M_0$ is the empty structure, and $\cC_0$
is the empty function, we have
$\Models_{L_0,L}^{\M_0} = \Models_{\cC_0, L} = \Models_L$.
If $L$ does have constant symbols, but $L_0$ and $\M_0$ are empty,
then we still have
\linebreak
$\Models_{L_0,L}^{\M_0} =  \Models_L$.

\begin{definition}
Let $L$ be a countable language and let $L_0$ be a sublanguage of $L$.
Let $\M_0$ be an $L_0$-structure on $\Nats$.
Then for every $\Lwow(L)$-formula $\varphi$,
define
\[\llrrsub{\varphi(\ell_1, \dots, \ell_j)}{\M_0}
\defas
 \{\M \in 
\Models_{L_0,L}^{\M_0} \st
\M\models \varphi(\ell_1, \dots, \ell_j)\}\]
for all $\ell_1, \ldots, \ell_j \in \Nats$, where $j\in\Nats$ is the number of free
variables (possibly $0$) of $\varphi$.
\end{definition}

When $\Models_{L_0,L}^{\M_0}$ is equipped with the 
$\sigma$-algebra consisting
of all such sets
\linebreak
$\llrrsub{\varphi(\ell_1, \dots, \ell_j)}{\M_0}$,
it also becomes a
standard Borel space.

\begin{definition}[{\cite[\S2.7]{MR1425877}}]
Let $L$ be a countable language and let $L_0$ be a sublanguage of $L$.
Let $\M_0$ be an $L_0$-structure on $\Nats$.
Define the \defn{relativized logic action}
\[\logicact^{\M_0} \colon \Aut(\M_0)\times
\Models_{L_0,L}^{\M_0}
\rightarrow
\Models_{L_0,L}^{\M_0}
\]
to be
the restriction of the action $\logicact\colon \sym\times \Models_L\rightarrow
\Models_L$.
\end{definition}

\section{Toy construction}
\label{simplified}
We now provide 
a toy
construction of invariant measures via
limits of finite structures, where the measure is concentrated on the
isomorphism class of a single graph.
This is  
a simplification of a special case
of 
the main construction of 
this
paper, which we present in order to illustrate several motivating ideas,
in a considerably easier setting.
This toy construction is also a variant of a special case of the main
construction of
\cite{AFP},
where it is shown
that whenever a countably infinite structure $\M$ in a
countable language $L$ has trivial definable closure, there is an
$\sym$-invariant measure on $\Models_L$ concentrated on
the isomorphism class of $\M$.

All graphs in this section will be simple graphs, i.e., undirected
unweighted graphs with no loops or multiple edges. Model-theoretically,
such a graph is considered to be a structure in the \emph{language of graphs,}
i.e., a language 
consisting of
a single binary relation symbol (interpreted as the edge relation), in which 
the
edge relation is symmetric
and irreflexive.

This 
toy
construction 
applies only to
the special case where the target structure is an
ultrahomogeneous
countably infinite 
graph
having trivial definable closure.
Admittedly, there are not many such structures: only a small number of parametrized classes of
countably infinite
graphs are ultrahomogeneous (see \cite{MR583847}),
and fewer still
have trivial definable closure (see, e.g., \cite{AFP}) ---
and even those have been treated before (essentially in \cite{MR2724668}).
However, 
this toy construction 
serves
to illustrate
some
of the
key ideas of the main construction.
In fact, the case of graphs 
is
particularly simple, because 
it allows us to
make  use of results from the theory of dense graph limits.

Roughly speaking, given a target countably infinite ultrahomogeneous graph, we
will build a sequence of finite graphs such that 
subgraphs sampled
from them (in
an appropriate sense) look more and more like induced ``typical'' subgraphs
of the target. Then the distribution of an appropriate limit of the random
graphs resulting from this sequence of sampling procedures
will constitute the invariant measure
concentrated on the isomorphism class of our target.

Our 
construction of the sequence of finite graphs 
resembles a
directed system of finite graphs.
This motivates  our main construction in
Section~\ref{invlimitconstruction}, which is 
built from
directed systems in a more
precise sense.

A key notion in the 
toy
construction will be that of ``duplication'', where\-by
a sequence of elements 
branches
into multiple copies that stand in
parallel
relationship to each other.
This notion, too, will be essential in the main construction.

Suppose $\M$ is a countably
infinite graph with underlying set $M$ that is a \Fr\ limit whose age
has the strong amalgamation
property; 
recall that for relational languages,
this
property is equivalent to $\M$ having trivial definable closure.

The strong amalgamation property
implies 
an
important property that we call \emph{duplication of quantifier-free
types}:
given any finite subset $A\subseteq M$ and any element 
$s\in M \setminus A$,
there is some 
$s'\in M\setminus A$  such that the quantifier-free type of
$A\cup \{s\}$ is the same as the quantifier-free type of 
$A \cup \{s' \}$.

As a consequence of this duplication property, for any $s_1, \ldots,
s_n \in M$, we can find sets $S_1, \ldots, S_n \subseteq \M$ of arbitrary
finite sizes such that each $s_i \in S_i$, and such that
for any tuple $s'_1, \ldots, s'_n$  satisfying $s'_i \in S_i$ for $1\le i\le n$,  the
quantifier-free type of $s'_1, \ldots, s'_n$ is the same as the
quantifier-free type of $s_1, \ldots, s_n$. We call the sequence $S_1,
\ldots, S_n$ a 
\defn{branching}
of $s_1, \ldots, s_n$, and say that each
$s_i$ 
branches
into $|S_i|$-many 
\defn{offshoots}.

\subsection{Convergence and graph limits}

As above, let $\M$ be an arbitrary countably infinite ultrahomogeneous graph
whose age has the strong amalgamation property.
We will construct a probability measure on countably infinite graphs with
underlying set $\Nats$ that is invariant under arbitrary permutations of
$\Nats$ and is concentrated on the isomorphism class of $\M$. We will do so
by constructing a sequence
$\<\M_i\>_{i\in\Nats}$ of 
graphs  of increasingly large finite size,
and considering the corresponding sequence of infinite random graphs
$\bigl\<\GG(\Nats, \M_i)\bigr\>_{i\in\Nats}$.

\begin{definition}
\label{LovDef}
Let $G$ be a finite graph.
The \defn{infinite random graph induced from $G$ with replacement}, written
$\GG(\Nats, G)$, is
a countably infinite random graph with underlying set $\Nats$ with edges
defined as follows.
Let $\<x_i\>_{i\in\Nats}$ be a sequence of elements of $G$ uniformly independently sampled
with replacement.
Then distinct $j, k\in\Nats$ have an edge between them  in $\GG(\Nats, G)$
precisely when $x_j$ and $x_k$ have an edge between them in $G$.
\end{definition}

This sampling procedure has arisen independently a number of times; see
\cite[\S10.1]{MR3012035} for some of its history.
The form we use can be concisely described using the theory of dense graph
limits, or \emph{graphons};
see \cite[\S11.2.2]{MR3012035} for details. 
That work describes,
given a graphon, a distribution on countably
infinite graphs  built from that graphon, called the \emph{countable random
graph model}. This distribution
corresponds to the distribution of $\GG(\Nats, G)$
in Definition~\ref{LovDef}
in the case where the graphon in question is the step-function built from
$G$ (\cite[\S7.1]{MR3012035}). Note, however, that 
this distribution
does not 
cohere
with
the definition in \cite[\S10.1]{MR3012035} of $\GG(k, G)$ for finite $k$ bounded by the number of
vertices of $G$, which involves sampling without replacement.

Our goal
is to find a sequence of finite graphs $\<\M_i\>_{i\in\Nats}$ as above, such that
the sequence of random variables $\<\GG(\Nats, \M_i)\>_{i\in\Nats}$ converges
in distribution to a random graph
that is almost surely isomorphic to $\M$, and
whose distribution is
invariant under permutations of $\Nats$.
The invariance will be automatic, as
each $\GG(\Nats, \M_n)$ is obtained via i.i.d.\ sampling, as described 
in the definition.
In order to show the convergence, we will use results from the theory of
graphons.

Given
a graph $G$, we
write $v(G)$ to denote the number of vertices of $G$.

\begin{definition}
Let $F, G$ be 
finite graphs. Let $k = v(F)$ and $n=v(G)$.
Then $\tind(F, G)$, the \defn{full homomorphism density},
is defined to be the 
fraction of maps
from $F$ 
to
$G$ that preserve both adjacency and non-adjacency, 
i.e., \[\tind(F, G) = \frac{\Full(F, G)}{n^k},\] where $\Full(F,G)$ is the number of homomorphisms from $F$
to $G$ that also preserve non-adjacency.
\end{definition}

The value
$\tind(F, G)$ may also be described in terms of the following random
procedure.
First consider
an independent random selection of
$v(F)$-many
vertices of $G$ chosen uniformly with replacement, each labeled with the corresponding
element of $F$.
(In particular, some vertices of $G$ may be labeled by
multiple vertices of $F$.)
Then $\tind(F, G)$ is the probability that
the 
graph with labels from $F$ induced by the sampling procedure
is a labeled copy of $F$,
preserving both edges and non-edges.

This notion of a full homomorphism
occurs in the graph homomorphism literature, e.g.,
in \cite[\S1.10.10]{MR2089014}.
Note, however, that $\tind$
is somewhat different from the various 
densities that are typically
used in the study of graph limits, 
namely, 
the density $t$ of homomorphisms,
$t_\mathrm{inj}$ of
injective homomorphisms,  and
$t_\mathrm{ind}$ of
induced injective homomorphisms, i.e., embeddings; for details see
\cite[\S5.2.2]{MR3012035}.

\begin{definition}
We say that a sequence of finite graphs $\<G_i\>_{i\in\Nats}$ is \defn{unbounded}
when  $\lim_{i\to\infty} v(G_i) = \infty$.
\end{definition}

The following definition of a type of convergence is slightly nonstandard as it uses $\tind$, but is
equivalent to the more usual definitions in the literature on dense
graph limits, which involve
the other density notions,  as described in
the discussion in the beginning of \cite[\S11.1]{MR3012035}.
\begin{definition}
An unbounded sequence of finite graphs
$\<G_i\>_{i\in\Nats}$
is \defn{convergent} when the sequence of induced subgraph densities
\[\bigl\<\tind(F, G_i)\bigr\>_{i\in\Nats}\]  converges for every finite graph $F$.
\end{definition}

\begin{theorem}[{\cite[Theorem~11.7]{MR3012035}}]
\label{LovMain}
Let $\<G_i\>_{i\in\Nats}$ be an 
unbounded sequence of finite graphs that is convergent.
Then  $\bigl\<\GG(\Nats, G_i)\bigr\>_{i\in\Nats}$ converges in distribution
to a countably infinite random graph whose distribution is
an $\sym$-invariant measure.
\end{theorem}

In fact, every such $\sym$-invariant measure is ergodic, as shown by Aldous \cite[Lemma~7.35]{MR2161313}; for an argument involving graph limits, see \cite[Proposition~3.6]{JGT:JGT20611}.

\begin{corollary}
\label{mainGraphonCor}
Let $\<G_i\>_{i\in\Nats}$ be an unbounded sequence of finite graphs.
Suppose the limiting probability
\[\lim_{i \to \infty} \Pr\bigl( \GG(\Nats, G_i) \models q(0,\ldots, \ell-1)\bigr)\]
exists for every quantifier-free type $q$ in the language of graphs,
where $\ell$ is
the number of free variables of $q$.
Then  $\bigl\<\GG(\Nats, G_i)\bigr\>_{i\in\Nats}$ converges in distribution
to an $\sym$-invariant measure on countably infinite graphs.
\end{corollary}
\begin{proof}
By Theorem~\ref{LovMain}, it suffices to show that
$\bigl\<\tind(F, G_i)\bigr\>_{i\in\Nats}$  converges for every finite graph $F$.

Let $F$ be an arbitrary finite graph with underlying set $\{0, \ldots,
n-1\}$, where $n = v(F)$.
Let $q_F$ be the unique non-redundant quantifier-free type with $n$-many free
variables such that
\[
F \models q_F(0, \ldots, n-1).
\]
Note that, for each $j\in\Nats$,
\[\tind(F, G_j) =
\Pr\bigl( \GG(\Nats, G_j) \models q_F(0, \ldots, n-1)\bigr).
\]
Hence
$\bigl\<\tind(F, G_i)\bigr\>_{i\in\Nats}$  converges, as
\[\Bigl\<\Pr\bigl( \GG(\Nats, G_i) \models q_F(0, \ldots, n-1)\bigr)\Bigr\>_{i\in\Nats}\]
converges by hypothesis.
\end{proof}

\subsection{Construction}

Because $\M$ is a \Fr\ limit in a finite relational language, 
as discussed in \S\ref{Frlimits}
we may 
take
its first-order theory $T$ 
to be axiomatized by
pithy $\Pi_2$ 
extension axioms, so that
\[T = \{(\forall \x)(\exists y)\varphi_i(\x,y) \st i \in \Nats\},\]
where each $\varphi_i$ is quantifier-free; we may further assume that
for each 
$i\in \Nats$
there are infinitely many indices
$j\in\Nats$ such that $\varphi_i = \varphi_j$.
We will consider, in
successive stages, each such 
formula $\varphi_i(\x,y)$
and 
every tuple
$\a\in\M$ of the same
length as $\x$, and will look for \defn{witnesses} in $\M$
to $(\exists y)\varphi_i(\a,y)$, i.e., instantiations $b\in \M$ of $y$ that make
$\varphi_i(\a, b)$ hold in $\M$.

Our construction proceeds in stages, at each of which we build a
finite structure larger than that in the previous stage.
We will think of the 
structure
that we build at stage $n$ as
consisting of $(n+1)$-many slices, each built at
a substage.
In the first substage
of stage $n$,
we add a slice that consists of new witnesses to the formula under
consideration (or one new element, if no witnesses are needed). In the remaining substages, we
branch
each
element of each
old slice 
into
some
number of 
offshoots.

Specifically, we divide each stage $n$ into $(n+1)$-many distinct
substages
indexed by pairs $(n, k)$, where $0 \le k \le n$.
The substage $(n, 0)$ involves adding witnesses to extension axioms for everything from stage $n-1$
(as one often does when iteratively building a \Fr\ limit).
The substages
$(n, k)$, for $0 < k \le n$,
consist of successively 
branching
elements.
By duplicating ever larger portions, we cause the
structure to asymptotically stabilize.

More precisely,
at substage $(n, k)$ we will define 
a
structure $\M_n^k$  and
a
set $B(n, n-k)$. \linebreak The intuition is that $B(n,n)$ consists of 
new witnesses,
while $B(n, n-k)$, for \linebreak $k>0$, consists of all
elements of $\M_n^k$ that are 
offshoots
of elements that first appear 
at
substage $(n-k, 0)$. In particular, the underlying set of
$\M_n^k$ will be 
\[
\bigcup_{i = 0}^{n-1-k} B(n-1, i) \ \ \cup \ \bigcup_{i = n-k}^n B(n, i),
\]
because
at substage $(n,k)$, the newly-constructed set
$B(n, n-k)$ contains all elements of
$B(n-1, n-k)$.
\ \\

\noindent \ul{Substage $(0,0)$:} Let $\M_0^0$ be any finite 
substructure
of the \Fr\ limit $\M$,
and let $B(0,0)$ be its underlying set.
\ \\

\noindent \ul{Substage $(n, 0)$, for $n> 0$:}
Let 
$\ell_n$
be one less than the number of free variables in the formula $\varphi_n$.
Let $A$ be the set of those $\a\subseteq \M_{n-1}^{n-1}$ of 
length
$\ell_n$
such that
\linebreak
$\M_{n-1}^{n-1}
\not \models \bigvee_{b \in \a}\varphi_n(\a,b)$. 
We now define $B(n, n)$ and $\M_n^0$. Consider whether or not $A$ is empty.

If $A$ is non-empty, then 
for each $\a \in A$ choose a distinct element
$d_\a\in \M$
that 
satisfies
$\M_n^0 \models \varphi_n(\a, d_\a)$.
We can always find such a collection of witnesses, because our formulas are
realized in the \Fr\ limit $\M$. Furthermore, because $\M$ has strong amalgamation,
by duplication of
quantifier-free types, we may assume that for any distinct tuples $\a, \a' \in
\M_{n-1}^{n-1}$, the elements $d_\a$ and $d_{\a'}$ are distinct.
Define \linebreak $B(n,n) = \{d_\a\st \a \in A\}$ and let
$\M_n^0$ be any substructure of $\M$
extending $\M_{n-1}^{n-1}$ by the elements of
$B(n,n)$.

If 
$A$ is empty,
then let $B(n, n)$ consist of an arbitrary single element of $\M$ not in $\M_{n-1}^{n-1}$, and set $\M_{n}^0$ to be 
the (unique) substructure of $\M$
extending $\M_{n-1}^{n-1}$ by the element of $B(n,n)$.
\ \\

\noindent \ul{Substage $(n, k)$ for $0 < k \le n$:}
Let $\alpha_n \defas  2^{n-1} \, |B(n, n)|$.
Let
$\M_n^k$ 
be any substructure of $\M$ that extends
$\M_n^{k-1}$ 
to some structure in which
each element 
of $B(n-1, n-k)$ branches
into
precisely
$\alpha_n$-many
offshoots, and these are the only new elements.
Let $B(n, n-k)$  be the set of those elements of $\M_n^k$ that are 
an offshoot
of some element of
$B(n-1, n-k)$.
By the definition of $\alpha_n$, we have
\[
\frac{|B(n, n)|}
{|\M^k_n|} \le 2^{-(n-1)}.
\]
This concludes the construction.
\ \\

For notational convenience, we will henceforth refer to $\M^n_n$ as $\M_n$.

In the verification, we will need a particular projection map.
Let $\widetilde{\pi}$ be the following map 
from the union of the underlying sets of all $\M_n$, for $n\in\Nats$, to 
itself.
The map $\widetilde{\pi}$ takes each element
of $B(n,n-k)$ to the element of 
$B(n-k, n-k)$
of which it is 
a $k$-fold
offshoot (i.e., an offshoot's offshoot's offshoot, etc., 
$k$
levels deep),
for $n\in\Nats$ and $0\le k < n$, and the
identity map on each 
$B(n, n)$.
This is well-defined because if an element of the domain is in both
$B(n, k)$ and $B(m, \ell)$,
then $k = \ell$.
(Note that $\widetilde{\pi}$ is not the same as
the projection map $\pi$ defined
in Section~\ref{preliminaries}, though it will play a similar role here
to that of $\pi$ in the main construction in
Section~\ref{invlimitconstruction}.)

\subsection{Verification}

We now show that the sequence of random graphs
$\bigl\<\GG(\Nats, \M_i)\bigr\>_{i\in\Nats}$
converges in distribution 
to a random graph that is almost surely isomorphic to our original graph $\M$.
We show this in two parts: convergence to such a random graph,
whose
distribution is an invariant measure on countable graphs, and concentration
of this invariant measure on the desired isomorphism class.

\begin{proposition}
\label{SimplifiedInvMeas}
The sequence of random graphs $\bigl\<\GG(\Nats, \M_i)\bigr\>_{i\in\Nats}$
converges in distribution 
to a countably infinite random graph whose distribution is
an $\sym$-invariant measure.
\end{proposition}
\begin{proof}
Note that
$\<\M_i\>_{i\in\Nats}$ is an unbounded sequence of finite graphs.
Hence by Corollary~\ref{mainGraphonCor}, it suffices to show that
\[
\Bigl\<\Pr\bigl(\GG(\Nats, \M_i) \models
q(0, \ldots, \ell-1)\bigr)\Bigr\>_{i\in\Nats}
\]
is Cauchy for every quantifier-free type $q$ in the language of graphs,
where
$\ell$ is the number of free variables of $q$.

Fix such a $q$ and $\ell$.
For each $n\in\Nats$,  define
\[
\delta_{n+1} \defas
\Pr\bigl(\GG(\Nats, \M_n) \models q(0, \ldots, \ell-1)\bigr)
- \Pr\bigl(\GG(\Nats, \M_{n+1}) \models q(0, \ldots, \ell-1)\bigr)
.
\]
We will show that $\delta_{n+1}$
decays exponentially
in $n$ for fixed $\ell$.

Let $G_n$ be a sample from
$\GG(\Nats, \M_n)$, and
let $\<a_i\>_{i\in\Nats}$ be the random sequence of vertices (with
replacement) chosen from $\M_n$ in the course of the sampling procedure.
Likewise, let
$G_{n+1}$ be a sample from
$\GG(\Nats, \M_{n+1})$  with vertex sequence
$\<b_i\>_{i\in\Nats}$. Observe that
\[\Pr\bigl(\GG(\Nats, \M_n) \models q(0, \ldots, \ell-1)\bigr) =
\Pr\bigl(G_n \models q(a_0, \ldots, a_{\ell-1})\bigr) \]
and
\[\Pr\bigl(\GG(\Nats, \M_{n+1}) \models q(0, \ldots, \ell-1)\bigr) =
\Pr\bigl(G_{n+1} \models q(b_0, \ldots, b_{\ell-1})\bigr).\]

Let $E_{n+1, \ell}$ be the event that
for each $i$ such that $0\le i \le \ell-1$,
the projection 
\linebreak $\widetilde{\pi}(b_i) \in \M_n$.
By our construction, the conditional probability
\[
\Pr\bigl(G_{n+1} \models q(b_0, \ldots, b_{\ell-1})\,\bigm |\, E_{n+1, \ell}\bigr)
\]
satisfies
\[
\Pr\bigl(G_{n+1} \models q(b_0, \ldots, b_{\ell-1})\,\bigm |\, E_{n+1, \ell}\bigr)
\ =\
\Pr\bigl(G_{n} \models q(a_0, \ldots, a_{\ell-1})\bigr)
.\]
Therefore $\delta_{n+1} \le 1- \Pr(E_{n+1, \ell})$.

By construction of $\M_{n+1}$, we have
\[
\Pr(E_{n+1, \ell}) = \Bigl(1-\frac{|B(n+1, n+1)|}{|\M_{n+1}|}\Bigr)^\ell
\]
Recall that  at the end of the construction we observed that
\[
\frac{|B(n+1,n+1)|}
{|\M_{n+1}|} \le  2^{-n},
\]
and so $\Pr(E_{n+1, \ell}) \ge (1- 2^{-n})^\ell$.
Using Bernoulli's inequality, we obtain the bound
$\Pr(E_{n+1, \ell}) \ge  1 - \ell \ 2^{-n}$,
and so
$\delta_{n+1} \le \ell \ 2^{-n}$,
as desired.
\end{proof}


Let $\mu_\M$ denote the distribution of the limit of $\bigl\<\GG(\Nats,
\M_i)\bigr\>_{i\in\Nats}$.
Proposition~\ref{SimplifiedInvMeas} demonstrates
that $\mu_\M$ is an
$\sym$-invariant  measure on $\Models_L$, where $L$ is the language of
graphs.
We now show that $\mu_\M$ assigns measure $1$
to the isomorphism class of $\M$. We begin with a combinatorial lemma.

Recall that
for each $j\in\Nats$, we have defined
$\ell_j\in\Nats$ to be one less than the number of free variables in
the quantifier-free formula $\varphi_j$.
For each $n, j\in\Nats$, define
\[\gamma_{n,j} \defas
\Pr\Bigl(
\GG(\Nats, \M_n) \models (\exists y)\, \varphi_j\bigl(0\cdots(\ell_j-1), y\bigr)
\Bigr)
.\]
Before proving our main bound on this quantity, we need a technical lemma.
\begin{lemma}
\label{veryCombLemma}
Let $k\in\Nats$ and suppose $0 < C < 2^k$. Then
\[\prod_{i = k}^\infty (1 -  C \ 2^{-i})  \ge
(1 - C\ 2^{-k})^2
.
\]
\end{lemma}
\begin{proof}
By our hypothesis on $C$, each term of the product is positive.
In particular, we have
\begin{eqnarray*}
\log \bigl(\prod_{i = k}^\infty (1 -  C \ 2^{-i})  \bigr)
= \sum_{i=k}^\infty \log (1 - C  \ 2^{-i})
\end{eqnarray*}
By the concavity of the function $\log(1-t)$, we have $\log(1-t) \ge t \, \log(1-t_0)/t_0$ for $t_0 \ge t > 0$. Setting $t_0 = C\, 2^{-k}$ and $t = C\, 2^{-i}$ where $i\ge k$, we obtain
\begin{eqnarray*}
\sum_{i=k}^\infty \log (1 - C  \ 2^{-i})
&\ge&
\sum_{i=k}^\infty 
C\, 2^{-i} 
\log (1 - C  \ 2^{-k}) / (C \, 2^{-k})
\\
&=& 
\log (1 - C  \ 2^{-k})\ 
\sum_{i=k}^\infty 
2^{-i+k} \\
&=&
2 \, \log (1 - C  \ 2^{-k}).
\end{eqnarray*}
Therefore  $\prod_{i = k}^\infty (1 -  C \ 2^{-i})  \ge (1 - C\ 2^{-k})^2$ by the monotonicity of $\log$.
\end{proof}

\begin{lemma}
\label{combLemma}
For all $j\in\Nats$,
\[\lim_{n \rightarrow \infty}\gamma_{n,j} =1.\]
\end{lemma}
\begin{proof}
By our enumeration of formulas $\varphi_j$, observe that 
for all $n\in\Nats$ there is some $j\ge 1$ such that
$\gamma_{n, j} = \gamma_{n, 0}$. Hence it suffices to prove the claim for all $j\ge 1$.

We may 
assume that $j$ is large enough that $\ell_j < 2^{j-1}$, as each formula is enumerated infinitely often, and $\gamma_{n,j}$ depends only on $n$ and on $\varphi_j$, not on $j$.

Fix such a $j\ge 1$.
As in the proof of
Proposition~\ref{SimplifiedInvMeas}, for $n\in\Nats$
let $G_n$ be a sample from
$\GG(\Nats, \M_n)$, and
let $\<a_i\>_{i\in\Nats}$ be the random sequence of vertices (with
replacement) chosen from $\M_n$ in the course of the sampling procedure.

Analogously, for $n>j$, define $D_{n, j, \ell_j}$ to be the event that
for each $i$ such that $0\le i \le \ell_j-1$,
the projection 
$\widetilde{\pi}(a_i) \in \M_{j-1}$.
Recall that
$\Pr(E_{h, \ell_j}) \ge 1 -  \ell_j \ 2^{-(h-1)}$ for each $h\in \Nats$,
and so 
\begin{eqnarray*}
\Pr(D_{n, j, \ell_j}) 
&\ge& \Pr(E_{n, \ell_j}) \cdot 
\Pr(E_{n-1, \ell_j}) \cdots
\Pr(E_{j, \ell_j}) \\ 
&\ge& 
(1 -  \ell_j \ 2^{-(n-1)}) \cdot
(1 -  \ell_j \ 2^{-(n-2)}) \cdots
(1 -  \ell_j \ 2^{-(j-1)}). 
\end{eqnarray*}
Taking $C = \ell_j$ and $k = j-1$ in Lemma~\ref{veryCombLemma}, we obtain
\[\prod_{i = j-1}^\infty (1 -  \ell_j \ 2^{-i})  \ge 
(1 - \ell_j\ 2^{-(j-1)})^2,
\]
and so
\[
\Pr(D_{n, j, \ell_j}) 
\ge
(1 - \ell_j\ 2^{-(j-1)})^2,
\]
as each term in the infinite product is between $0$ and $1$.

Now let $F_{n, \ell_j}$ be the event that 
the elements
$a_0, \ldots, a_{\ell_j-1}$ 
of $\M_n$ are distinct. Observe that, because $\<\M_i\>_{i\in\Nats}$
is an unbounded sequence of graphs,
\[
\lim_{n\to \infty} \Pr(F_{n, \ell_j}) = 1.
\]

Because of 
the way
witnesses are chosen at
substage 
$(j, j)$, for $n>j$
if events $D_{n, j, \ell_j}$ and
$F_{n, \ell_j}$ hold, then
\[
\GG(\Nats, \M_n) \models (\exists y) \varphi_j\bigl(0\cdots(\ell_j-1), y\bigr).
\]
Therefore,
\[
\Pr\Bigl(
\GG(\Nats, \M_n) \models (\exists y)\, \varphi_j\bigl(0\cdots(\ell_j-1), y\bigr)
\Bigr)
\ge
\Pr(D_{n, j, \ell_j}) \cdot \Pr(F_{n, \ell_j}).
\]
For $n>j$, define $\zeta(n,j)$ to be the greatest $k<n$ such that $\varphi_k = \varphi_j$.
We then have
\begin{eqnarray*}
\Pr\Bigl(
\GG(\Nats, \M_n) \models (\exists y)\, \varphi_j\bigl(0\cdots(\ell_j-1), y\bigr)
\Bigr)
&\ge&
\Pr(D_{n, \zeta(n,j), \ell_j}) \cdot 
\Pr(F_{n, \ell_j})\\
&\ge&
(1 - \ell_j\ 2^{-(\zeta(n, j)-1)})^2 \cdot
\Pr(F_{n, \ell_j}).
\end{eqnarray*}
Because each formula is enumerated infinitely often, $\lim_{n\to\infty} \zeta(n,j) = \infty$. Hence
\[
\lim_{n\to\infty}
\Pr\Bigl(
\GG(\Nats, \M_n) \models (\exists y)\, \varphi_j\bigl(0\cdots(\ell_j-1), y\bigr)
\Bigr)
= 1,
\]
as desired.
\end{proof}

\begin{proposition}
The $\sym$-invariant measure  $\mu_\M$ is concentrated on the isomorphism class of
$\M$.
\end{proposition}
\begin{proof}
For each $j\in\Nats$, we have
\[
\mu_\M\Bigl(
\bigllrr{(\exists y) \varphi_j
\bigl(0\cdots(\ell_j-1), y\bigr)}
\Bigr)
 = 1,
\]
by Lemma~\ref{combLemma}. Therefore, by the $\sym$-invariance of
$\mu_\M$,
we have
\[
\mu_\M\Bigl(
\bigllrr{(\forall\x)(\exists y) \varphi_j
\bigl(\x, y\bigr)}
\Bigr)
 = 1,
\]
where $\x$ is an $\ell$-tuple of distinct variables. But $T$ consists
solely of sentences of the form $(\forall\x)(\exists y) \varphi_j \bigl(\x, y\bigr)$.
Hence $\mu_\M$ is
concentrated on the class of models of $T$.
Because $T$ is $\aleph_0$-categorical and $\M \models T$, the
measure $\mu_\M$ is concentrated on the isomorphism class of $\M$.
\end{proof}


\section{Inverse limit construction}
\label{invlimitconstruction}

We now give the key technical construction of the paper. This will take a theory with
certain properties and produce a probability measure, invariant under
permutations of the non-constant elements in the underlying set,
that is concentrated on the class of models of the theory. This construction is a
variant of the one in \cite{AFP} and will be the crucial tool used in 
later sections.

\subsection{Setup}
Before providing the construction itself, we describe the main conditions
it requires. We
begin by fixing the following languages, theories, and quantifier-free types.

First let
$\<L_i\>_{i\in\Nats}$
be an increasing sequence of countable languages having no function
symbols, but possibly both constant and relation symbols, and
let $L_{\infty} \defas \bigcup_{i \in \Nats}L_i$, so that
\[
L_0 \subseteq
L_1 \subseteq
L_2 \subseteq
\cdots \subseteq  L_\infty
.
\]
Further assume that all constant symbols appearing in any $L_i$ are
already in the language $L_0$; call this set of constant symbols $C$.

Now fix an increasing sequence
$\<T_i\>_{i \in \Nats}$
of
countable pithy $\Pi_2$
theories
that are quantifier-free complete and satisfy
$T_i\in \Lwow(L_i)$ for each $i \in \Nats$.
Let $T_{\infty} \defas \bigcup_{i \in \Nats}T_i$, so that
\[
T_0 \subseteq
T_1 \subseteq
T_2 \subseteq
\cdots \subseteq  T_\infty
.
\]

For each $i\in\Nats$, let $Q_i = \<q_{j}^i\>_{j \in \Nats}$ be 
any
sequence
of complete
non-constant
quantifier-free $L_i$-types 
that are consistent with $T_i$
and that satisfy
the following four conditions. Let
$k^i_j$ denote the number of free variables of $q^i_j$.
\begin{itemize}
\item[\textbf{(W)}] For each $i,j \in\Nats$  and every sentence $(\forall
\x)(\exists y)\psi(\x,y) \in T_i$  for which $|\x| = k^i_j$,
there is some $e_{i, j, \psi} \in \Nats$
such that  $\qn \defas q^i_{e_{i, j, \psi}}$ is a quantifier-free type
with one more free variable than $q^i_j$ and such that
\begin{eqnarray*}
&\models&
(\forall \x, y) \bigl(
\qn(\x,y)\rightarrow q_{j}^i(\x)
\bigr)
\qquad \text{~and} \\
&\models&
(\forall \x, y) \bigl(
\qn(\x,y)\rightarrow \psi(\x,y)
\bigr).
\end{eqnarray*}

\item[\textbf{(D)}] For each $i, j \in \Nats$ and variable $y$
such that  $q^i_j$ is a non-redundant quantifier-free type satisfying
\[
\models
(\forall \x, y) \bigl(
q_j^i(\x,y)\rightarrow \bigwedge_{c \in C} (y \neq c)
\bigr),
\]
where $|\x| + 1 = k^i_j$,
there is some $f_{i, j}$ such that
the quantifier-free type
$\qn \defas q^i_{f_{i, j}}$ has $(k^i_j + 1)$-many
free variables, is non-redundant, and satisfies
\[
\models
(\forall \x, y, z) \bigl(
\qn
(\x,y,z) \rightarrow \bigl(q_{j}^i(\x,y)\And q_{j}^i(\x,z)\bigr)
\bigr).
\]

\item[\textbf{(E)}] For each $i, j \in \Nats$ there is some
$j'\in\Naturals$ such that
\[
\models
(\forall \x) \bigl(
q^{i+1}_{j'}(\x)\rightarrow
q_{j}^{i}(\x)
\bigr),
\]
where $|\x| = k^i_j = k^{i+1}_{j'}$.

\item[\textbf{(C)}] For each $i, j \in \Nats$ and quantifier-free type $p$
such that
\[
\models
(\forall \x, \ww) \bigl(
q^i_j(\x) \to p(\ww)
\bigr),
\]
where $|\x| = k^i_j$ and $\ww$ is a subtuple of variables of $\x$ with $|\ww|$ equal to the number of free variables of $p$,
there is some $h_{p}\in\Nats$ such that
$\qn \defas q^i_{h_{p}}$ satisfies
\[
\models
(\forall \ww) \bigl(
\qn(\ww)
\leftrightarrow  p(\ww)
\bigr).
\]
\end{itemize}

Condition (W) ensures that for each quantifier-free type in $Q_i$ and
pithy $\Pi_2$ 
sentence
in the theory $T_i$, the sequence $Q_i$ contains some
extension of the quantifier-free type that \emph{witnesses} the formula.

Condition (D) requires that for every non-redundant quantifier-free type in
$Q_i$ and every free variable of that quantifier-free type which 
it requires to not be instantiated by
a constant,
there is some other
quantifier-free type in $Q_i$ that \emph{duplicates} that
variable. In particular, by repeated use of (D), we can show that for any
non-redundant $q^i_j \in Q_i$, any $h\in\Naturals$, and any $k^* \le k^i_j$
such that
\[
(\forall \x, \y) \bigl(q^i_j(\x, \y) \to
\bigwedge_{z \in \y}
\bigwedge_{c \in C}
(z \ne c) \bigr)
\]
where $|\x| = k^i_j - k^*$ and $|\y| = k^*$,
there is an
$f^*_{i,j,k^*}\in \Nats$ such that the quantifier-free type
$\qn \defas q^i_{f^*_{i, j, k^*}}$ has $(k_{i,j} + h\,k^*)$-many
free variables, and
for all functions
\[\beta\colon \{1, \ldots, k^*\}\rightarrow \{0, \ldots, h\},\]
we have
\begin{eqnarray*}
\models
(\forall \x, \,
y_1^0 \cdots y_1^h \,\cdots \,y_{k^*}^0 \cdots y_{k^*}^h
)
\ 
 \bigl (
\qn
(\x,
y_1^0 \cdots y_1^h \, \cdots \, y_{k^*}^0 \cdots y_{k^*}^h
)
 \rightarrow
q_{j}^i(\x,
y_1^{\beta(1)} \ldots y_{k^*}^{\beta(k^*)}
)\bigr),
\end{eqnarray*}
where
the $y_\ell^w$ are new distinct variables,
for $1 \leq \ell \leq k^*$ and $0\le w\le h$.

In summary,
$\qn$  is
a quantifier-free type
that duplicates, $(h+1)$-fold, all variables of $q^i_j$.
Furthermore,
for every 
tuple
of variables from $\qn$ 
that contains exactly one duplicate of each
variable of $q^i_j$,
the resulting
restriction of $\qn$ to those variables is
precisely $q^i_j$ with corresponding variables substituted. We call such a $\qn$ an \defn{iterated duplicate}.
Recall our assumption that each $T_i$ is quantifier-free complete, hence
consistent, and that each element of $Q_i$ is consistent with $T_i$.
Therefore, as a consequence of iterated duplication, each $T_i$ must have models
with infinitely many elements that do not instantiate constant symbols.

Condition (E) says that for every quantifier-free type in $Q_i$ and larger language, we can find an
\emph{extension} of that quantifier-free type to that language.

Condition (C) says that the quantifier-free types of $Q_i$ are
\emph{closed} under implication.

There will be one further condition, which we will not always require to hold.
However, when it
does hold, 
it will
guarantee that the
construction assigns measure $0$ to every isomorphism class of models of
the 
target
theory.

\begin{itemize}
\item[\textbf{(S)}] For some $\ell \in\Nats$ (called the \emph{order} of
splitting), every $i \in \Nats$, and
every non-redundant quantifi\-er-free $L_i$-type $q^i_j \in Q_i$ with
$k^i_j \geq \ell$, there is some  $e \in \Nats$ and some quantifier-free
$L_{e}$-type $\qn \in Q_{e}$ with $2k^i_j$
many free variables, such that 
for each $\beta\colon  \{1, \ldots, k^i_j\} \to \{0, 1\}$, we have
\[\models (\forall x^0_1x^1_1 \dots
x^0_{k^i_j}x^1_{k^i_j})\bigl(\qn(x^0_1x^1_1 \dots x^0_{k^i_j}x^1_{k^i_j})
\rightarrow q^i_j(x^{\beta(1)}_1 \dots x^{\beta(k^i_j)}_{k^i_j})\bigr),
\]
where
$x^0_1x^1_1 \cdots x^0_{k^i_j}x^1_{k^i_j}$ is a tuple of distinct free
variables,
and for each $i_1, \ldots, i_\ell \in \Nats$ such that $1\leq i_1 < i_2 <\dots
< i_\ell  \leq k^i_j$, and each
$\gamma_0, \gamma_1 \colon \{1, \ldots, \ell\} \to \{0,1\}$, there are distinct
non-redundant $p_0, p_1 \in Q_j$ such that for $w \in \{0, 1\},$
\[\models (\forall x^0_1x^1_1 \dots
x^0_{k^i_j}x^1_{k^i_j})\bigl(\qn(x^0_1x^1_1 \dots x^0_{k^i_j}x^1_{k^i_j})
\rightarrow p_w(x^{\gamma_w(1)}_{i_1} \dots x^{\gamma_w(\ell)}_{i_\ell})\bigr).
\]
We call $\qn$ a \emph{splitting} of $q^i_j$ of order $\ell$.
\end{itemize}

If there is such an $\ell$ then we say that $\<Q_i\>_{i \in \Nats}$ has
\defn{splitting of quantifier-free types of order $\ell$}.

The intuition is that if (S) is satisfied then for every non-redundant
quantifier-free type in $\bigcup_{i \in \Nats} Q_i$ in at least $\ell$-many free
variables, 
there is some larger language in which we can duplicate the
quantifier-free type so that every quantifier-free subtype with $\ell$-many free variables splits into at least two
distinct quantifier-free types. In other words, for each quantifier-free subtype with
$\ell$-many free variables, if we consider all ways in which it is duplicated
(i.e., all the quantifier-free types where no two distinct free variables are duplicates of
the same variable), then that collection of quantifier-free subtypes always has at least two elements.
We will use this condition to show that 
any given
quantifier-free $\ell$-type
is
realized with probability $0$.

We now turn to the construction itself.

\subsection{Construction}
\label{construction-subsec}

The aim is to construct a continuum-sized measurable space with certain properties.
This will proceed via
the inverse limit of a system of finite structures in an increasing system
of languages with
associated measures.
We will build this system of structures in stages,
each of which will interleave four tasks.
The first task is to enlarge the underlying set and update the
measures so that they assign
mass to the new set in a way that is compatible with our earlier choices.
The second task is 
to add 
elements to ensure that ever more of our pithy $\Pi_2$
theory is realized, and 
adjust
the mass accordingly. The third task is to
make sure the quantifier-free type of the entire structure up until this
point is duplicated. This will ensure that the end result is a continuum-sized
structure. Finally, the fourth task is to ensure that if there is
splitting of quantifier-free types of some order $\ell$, then
the appropriate quantifier-free type splits
as we enlarge the language. This will ensure that
we obtain a continuum-sized structure and a measure such that under a certain sampling procedure,
the probability of any particular quantifier-free type
with $\ell$-many free variables being realized is 0, and hence sampling from
our 
structure will not assign positive measure to the isomorphism class of any single structure.

The construction proceeds in stages indexed by
$n\in \Nats \cup \{\infty\}$.
At each finite stage, the structure we construct will have underlying set equal to
the union of
a fixed countable set 
$C$
of elements that instantiate the constant symbols
with
some finite subset of $\wlw = \bigcup_{i\in\Nats} \Nats^i$.
Recall the infinitary theory
$T_{\infty} = \bigcup_{i \in \Nats}T_i$.
in the language
\linebreak
$L_{\infty} = \bigcup_{i \in \Nats}L_i$.
Fix an enumeration
$\<\varphi_i(\x_i, y)\>_{i\in\Nats}$
of all (quantifier-free) \linebreak $\Lwow(L_\infty)$-formulas
such that
each formula occurs infinitely often, and such that for each $i\in\Nats$,
\[
(\forall \x_i)(\exists y)\varphi_i(\x_i,y) \in T_\infty
\]
and
the formula $\varphi_i$ has precisely $(|\x_i| +1)$-many free variables; let
$\xi_i$ denote 
$|\x_i|$.
Let $\<\ov{a_i}\>_{i \in \Nats}$ be an enumeration with repetition of
finite tuples of elements of $\wlw$
such that for all $i \in\Nats$, we have $|\ov{a_i}| = \xi_i$
and for every $\ov{a}\in (\wlw)^{\xi_i}$, there are infinitely
many $j$ such that $\varphi_j = \varphi_i$ and $\ov{a_j} = \ov{a}$.
Also fix an arbitrary non-degenerate probability measure $m^*$
on $\Nats$, i.e., such that no element has measure $0$.

At the end of each finite Stage $n\in\Nats$, we will have constructed
\begin{itemize}
\item a finite set $X_n \subseteq \Nats^{2n}$ such that $\pi^2(X_n) \supseteq
X_{n-1}$ (when $n\ge 1$),

\item a 
measure $m_n$ on $X_n$,

\item some natural number $\alpha_n >\alpha_{n-1}$ (when $n\ge 1$),

\item the complete non-redundant quantifier-free $L_{\alpha_n}$-type of
$X_n$ (chosen from $Q_{\alpha_n}$), and

\item an $L_{\alpha_n}$-structure $\X_n$.
\end{itemize}
In fact, $X_0$ will be empty and $\alpha_0 = 0$. We will define an
$L_0$-structure
$\X_0$, whose underlying set will be precisely a set of instantiations of the constant symbols
in $L_0$. Call this set of instantiations $C_0$.

For all $n\in\Nats$, the $L_{\alpha_n}$-structure $\X_n$ will have
underlying set $X_n \cup C_0$, and hence is determined by the
quantifier-free $L_{\alpha_n}$-type of $X_n$. We call $X_n$ the
\defn{constantless} part of $\X_n$.

For convenience of various indices, Stage $1$ will not add anything essential to those objects constructed in Stage $0$.

For $n \ge 2$ we will divide Stage $n$ into 
substages $n.i$, indexed by
$i \in \{0, 1, 2, 3\}$, each devoted to a different task:
$n.0$ (adding mass), $n.1$ (adding witnesses), $n.2$ (duplication of quantifier-free types), and $n.3$ (expanding the language).

At the end of Stage $n.i$, for $i\in \{0,1,2\}$, we will have constructed
\begin{itemize}
\item a finite set $X^i_n$,

\item a 
measure $m^i_n$ on $X^i_n$,

\item the complete non-redundant quantifier-free $L_{\alpha_{n-1}}$-type of
$X^i_n$ (chosen from \linebreak $Q_{\alpha_{n-1}}$), and
\item an $L_{\alpha_{n-1}}$-structure $\X^i_n$.
\end{itemize}
As with the major stages, each $L_{\alpha_{n-1}}$-structure $\X^i_n$ will have
underlying set $X^i_n \cup C_0$, and hence 
will be
determined by the
quantifier-free $L_{\alpha_{n-1}}$-type of $X^i_n$. We similarly call $X^i_n$ the
constantless part of $\X^i_n$.
Because each Substage $n.3$ completes Stage $n$, we write $\X_n$, $X_n$,
and $m_n$ rather than $\X^3_n$, $X^3_n$, and $m^3_n$, respectively.

Furthermore, 
the sets
will satisfy
\begin{itemize}
\item $X^0_n \subseteq X^1_n \subseteq \Nats^{2n-2}$,
\item $X^2_n \subseteq \Nats^{2n-1}$ and $\pi(X^2_n) \supseteq X^1_n$, and
\item $X_n \subseteq \Nats^{2n}$ and $\pi(X_n) \supseteq X^2_n$.
\end{itemize}

Finally, at the end of Stage $\infty$, we will have constructed an
$L_\infty$-structure $\X_\infty$ defined by the quantifier-free
$L_\infty$-type of each finite subset of the
\emph{infinite}
constantless part $X_\infty\subseteq \wtow$ of $\X_\infty$, and a probability measure $m_\infty$ on $X_\infty$.
The structure $\X_\infty$ may be viewed as a sort of inverse limit of the
structures $\X_n$ for $0 \le n < \infty$, with elements ``glued together''
in accordance with the projection map $\pi$.

We will also, at the end of each (sub)stage, verify that the new choices
cohere with those made earlier. Specifically,
they will satisfy
the following \emph{existence} and \emph{duplication} properties for every $j\in \Nats$:
\begin{itemize}
\item[($\mathscr{E}$)]
If $\varphi_{j+1}\in\Lwow(L_{\alpha_j})$, then
for every
tuple $\s = s_1, \ldots, s_{|\ov{a_{j+1}}|}$ of (not necessarily distinct) elements from $X_{j}$
such that
$\ov{a_{j+1}} \sqsubseteq \s$,  and every
$\ell_1, \ldots, \ell_{|\ov{a_{j+1}}|} \in \Nats^2$
such that
$s_1\^\ell_1, \ldots, s_{|\ov{a_{j+1}}|}\^\ell_{|\ov{a_{j+1}}|} \in X_{j+1}$,
we have
\[\X_{j+1} \models (\exists y)\varphi_{j+1}(
s_1\^\ell_1, \ldots, s_{|\ov{a_{j+1}}|}\^\ell_{|\ov{a_{j+1}}|} , y).\]

\item[($\mathscr{D}$)] For all $g\in\Nats$, all distinct $s_1, \ldots, s_g \in \Nats^{2j}$,
all $\ell_1, \ldots, \ell_g \in \Nats^2$,
and all quantifier-free  $L_{\alpha_j}$-types $r$
with $g$-many free variables, if $s_1, \ldots, s_g \in X_j$ and
\linebreak
$s_1\^\ell_1, \ldots, s_g\^\ell_g \in X_{j+1}$ then
\[ \X_j \models r(s_1, \ldots, s_g) \]
if and only if
\[\X_{j+1} \models r(s_1\^\ell_1, \ldots, s_g\^\ell_g).\]
\end{itemize}

Furthermore, for any $s \in X_j$, we have
\[m_j(s) =
m_{j+1}\bigl((\pi^2)^{-1}(s)\cap X_{j+1}\bigr)\]
 and
\[
\lim_{i \rightarrow \infty}m_i(X_i) =1.\]  In this sense, mass is preserved
via projection throughout the construction.

We now make the construction precise.

\vspace{10pt}
\noindent \ul{Stage 0:} Defining the mass on $\Nats$ and the
quantifier-free type of the constants.

We begin by defining the constantless part
$X_{0} \defas \emptyset$. Let $\alpha_0 \defas 0$.
Let $m_0$ be the unique measure on $X_0$, i.e., which satisfies $m_0(\emptyset) = 0$.

Choose
an arbitrary
element of $Q_0$ having no free variables.
Because $T_0$ is quantifier-free complete, there is only one such choice of
quantifier-free $L_0$-type (up to equivalence).
This quantifier-free type describes which relations hold of any finite tuple of
elements instantiating constant symbols. In particular, this determines
when two constant symbols must be instantiated by the same element.
Let $\X_0$ be an $L_0$-structure in which $X_0$ has this quantifier-free type, which
amounts to choosing a set of instantiations of the constant symbols,
related in this way. Let $C_0$ denote this set of instantiations,
and let $\cC_0$ be the map that assigns each constant symbol of $L_0$ to its
instantiation in $\X_0$.

\vspace{10pt}
\noindent 
\ul{Stage 1:} Same as stage 0.

Let $X_1 \defas X_0 = \emptyset$, let $\alpha_1 \defas 1$, and let $m_1$ be the unique measure on $X_1$.
Let $\X_1$ be the unique $L_1$-structure whose reduct to $L_0$ is $\X_0$.

\vspace{10pt}
\noindent \ul{Stage $n.0$ (for $1 < n < \infty$):} Adding mass.

Having already determined the $L_{\alpha_{n-1}}$-structure $\X_{n-1}$
and the measure $m_{n-1}$, we now define
an $L_{\alpha_{n-1}}$-structure $\X^0_n$ extending
$\X_{n-1}$,
and the associated measure
$m^0_{n}$.
We will define the structure $\X^0_n$ by choosing its constantless part
$X^0_n \supseteq X_{n-1}$ and the quantifier-free $L_{\alpha_{n-1}}$-type
of $X^0_n$.

This substage adds new elements of $\Nats^{2(n-1)}$ to the support
of $m_{n-1}$
so as to ensure that the eventual measure $m_\infty$ will be a probability measure.

If there is an $x \in X_{n-1}$ with $x = n\^\b$ for some $\b\in
\Nats^{2n-3}$, then let $\X_n^0 \defas \X_{n-1}$ be the same
$L_{\alpha_{n-1}}$-structure,
and let $m_n^0 \defas m_{n-1}$.

Otherwise let $X_n^0 \defas X_{n-1} \cup \{n\^0^{2n-3}\}$ and fix some
ordering on it.
Let \linebreak $q(\x, y) \in Q_{\alpha_{n-1}}$ be a quantifier-free type with $|X_n^0|$-many
free variables such that  if $q^*$ is the quantifier-free type of $X_{n-1}$
(considered as an increasing tuple in the corresponding ordering)
in $\X_{n-1}$,
then
\[\models (\forall \x, y) \bigl(q(\x,y) \rightarrow q^*(\x) \bigr).\]
Note that such a $q$ exists in $Q_{\alpha_{n-1}}$ by condition (D).
Define the quantifier-free \linebreak $L_{\alpha_{n-1}}$-type of $X^0_n$ in $\X^0_n$
(where $X^0_n$ is considered as an increasing tuple in that ordering)
to be $q$.
Finally, let $m_n^0(z) = m_{n-1}(z)$ for all $z\in X_{n-1}$ and
$m_n^0(n\^0^{2n-3}) = m^*(n)$, where $m^*$ is the non-degenerate
probability measure on $\Nats$ that we fixed before the construction.

In summary, at stage $n.0$, if no element of $X_{n-1}$ is a sequence
beginning with $n$,
then we add one such sequence to our set, adjust the measure accordingly,
and define the larger quantifier-free type appropriately.
Note that $\X_{n-1}$ is a substructure of $\X^0_n$, and so the
quantifier-free type of any tuple in $\X_{n-1}$ is the same as its
quantifier-free type in $\X^0_n$. Furthermore, it is clear from the
definition of $m_n^0$
that the measures $m_n^0$ and $m_{n-1}$ agree on elements in the
intersection of their domains.

\vspace{10pt}
\noindent \ul{Stage $n.1$ (for $1 < n < \infty$):} Adding witnesses.

We now extend $\X^0_n$ to $\X^1_n$, in particular defining the quantifier-free
$L_{\alpha_{n-1}}$-type of its constantless part $X^1_n\supseteq X^0_n$ so as to ensure that certain subtuples have witnesses to
appropriate formulas, and define the associated measure $m^1_n$.

Call $\varphi_n(\a_n, y)$ \defn{valid} for Stage $n$ when the following hold:
\begin{itemize}
\item $(\forall \x_n)(\exists y)\varphi_n(\x_n, y) \in \bigcup_{1\le i \leq n-1}T_{\alpha_i}$.

\item At least one tuple $\b$ of elements of $X_{n}^0$ satisfies
$\ov{a_n} \sqsubseteq \b$.

\end{itemize}
If $\varphi_n(\a_n,y)$ is not valid for Stage $n$ then do nothing.
Otherwise let $V$ be the set of all $\b\in X_n^0$ such that
$\ov{a_n} \sqsubseteq \b$ and
\[
\X_n^0 \models
\neg \bigvee_{d \in \b} \varphi_n(\b,d).
\]
 For each $\b \in V$, let $n_\b\in\Nats$ be such that for all
$x \in X_n^0$ we have $n_\b\not \preceq x$. Then let 
$X_n^1 \defas X_n^0\cup
\{n_\b\^0^{2n-3} \st \b \in V\}$.
Fix some ordering of $X^0_n$, and
let $q^*$ be the quantifier-free $L_{\alpha_{n-1}}$-type  of $X^0_n$
(considered as an increasing tuple under this ordering) in $\X^0_n$.
Choose a quantifier-free $L_{\alpha_{n-1}}$-type
$q \in Q_{\alpha_{n-1}}$
such that
if $q$ holds of $X_n^1$ under some ordering, then
\[
q^*(X_n^0) \And  \bigwedge_{\b \in V}\varphi_n(\b,
n_\b\^0^{2n-3})\]
holds, where
$X^0_n$ occurs in its ordering.
Note that
the formula $(\forall \x_n)(\exists y) \varphi_n(\x_n, y)$ is in $T$
and
$q^*$ is consistent with $T$.
Hence by condition (W) we can always find such a $q$.
Declare $q$ to be the quantifier-free $L_{\alpha_{n-1}}$-type of $X^1_n$ in
$\X^1_n$ (under that ordering).
In other words, we require that
either there is a new witness or some witness already existed.

Finally,
let $m_{n}^1$ agree with $m_n^0$ on $X_n^0$ and set $m_n^1(n_{\b}\^0^{2n-3}) \defas
m^*(n_{\b})$ for \linebreak $n_{\b}\^0^{2n-3}\in X_n^1\setminus X_n^0$.

At this substage, we have ensured that if $\varphi_n$ is valid (for stage
$n$) then
there are witnesses in $\X^1_n$ to $(\exists y)\varphi_n(\b, y)$
for all appropriate elements $\b$ of $X^0_n$.
We will use this fact
to verify property ($\mathscr{E}$) at the end of
stage $n$.

Again, $\X^0_{n}$ is a substructure of $\X^1_n$, and so the
quantifier-free type of any tuple in $\X^0_{n}$ is the same as its
quantifier-free type in $\X^1_n$. Likewise, it is clear from the
definition of $m_n^1$
that the measures $m_n^1$ and $m_{n}^0$ agree on elements in the
intersection of their domains.

\vspace{10pt}
\noindent
\ul{Stage $n.2$ (for $1 < n < \infty$):} Duplication of Quantifier-Free
Types.

Having defined $\X^1_n$
in the previous
substage, we now define $\X^2_n$, in which we duplicate the quantifier-free
type of $X^1_n$
in $\X^1_n$.
We will define the structure $\X^2_n$ by choosing its constantless part
$X^2_n \supseteq X^1_{n}$ and the quantifier-free $L_{\alpha_{n-1}}$-type
of $X^2_n$.
We also define the
associated measure $m^2_{n}$.

Let $\Lambda_n\in\Naturals$ be  large enough that if $n$ balls are placed
uniformly independently in $\Lambda_n$-many boxes, the probability of
two or more balls landing in the same box is less than $2^{-n}$.

Define
\[
X^2_n \defas
\bigcup_{1 \le j \leq \Lambda_n}\{x\^ j \st x \in X^1_n\},
\]
and fix an ordering of $X^2_n$.
Fix an ordering $\<x_i\>_{1\le i\le |X^1_n|}$ of the elements of $X^1_n$,
and let 
$q \in Q_{\alpha_{n-1}}$
be the quantifier-free type of this tuple.
Choose a quantifier-free type
$q^* \in Q_{\alpha_{n-1}}$
such that whenever
$q^*$ holds of
$X^2_n$ (under its ordering) and
any subset $\{y_i \st 1\le i\le |X^1_n|\}\subseteq
X^2_n$  satisfies
$\pi(y_i) = x_i$ for all $i$ such that $1\le i\le |X^1_n|$, then
\[
q\bigl(\<y_i\>_{1\le i\le |X^1_n|}\bigr)
\]
holds.
Recall that
our assumption (D) of duplication of quantifier-free types implies the
existence of iterated duplicates.
Hence there is such a $q^*$,
as it is precisely an iterated duplicate of $q$.
Declare $q^*$ to be the quantifier-free type of $X^2_n$
(under its ordering)  in $\X^2_n$.

Suppose that
$g\in\Nats$ and  $s_1, \ldots, s_g \in X_n^1$ are distinct.
Further suppose that
\linebreak
$\ell_1, \ldots, \ell_g \in \Nats$
such that
$s_1\^\ell_1, \ldots, s_g\^\ell_g \in X^2_{n}$.
Then note that for any quantifier-free $L_{\alpha_{n-1}}$-type $r$
with $g$-many free variables
\[ \X^1_n \models r(s_1, \ldots, s_g) \]
if and only if
\[\X^2_{n} \models r(s_1\^\ell_1, \ldots, s_g\^\ell_g).\]
This is the analogue,
for the situation of moving from substage $n.1$ to $n.2$,
of property ($\mathscr{D}$).

Finally, for each $x \in X^2_n$, define $m^2_n(x)  \defas
m^1_n(\pi(x))/{\Lambda_n}$.
In other words, for each $y \in X^1_n$, its mass is divided evenly between
its
$\Lambda_n$-many extensions.

\vspace{10pt}
\noindent
\ul{Stage $n.3$ (for $1 < n < \infty$):} Expanding the Language.

Having defined $\X^2_n$
in the previous
substage, we now define $\X_n$ itself, some $\alpha_n > \alpha_{n-1}$,
and the associated measure $m_n$. We will define by $\X_n$ via its
constantless part $X_n \supseteq X^2_n$ and
the quantifier-free $L_{\alpha_n}$-type of $X_n$.
We do this in a way that ensures that if,
for some $\ell\in\Nats$ such that $\ell \le |X^2_n|$, there is splitting of quantifier-free types of
order $\ell$, then for the least such $\ell$, as we enlarge
the language we split all non-redundant quantifier-free types with $\ell$-many free variables.

Fix some ordering on $X^2_n$ and let $p_{n-1}$ be the
quantifier-free
$L_{\alpha_{n-1}}$-type of $X^2_n$ (considered as an increasing tuple under
that ordering) in $\X^2_n$.

\noindent{\textbf{Case (a):}}
If either there is no splitting of quantifier-free types of order $\ell$ for any
$\ell \in \Nats$, or there is a splitting, but the least such order $\ell$
is greater than $|X^2_n|$,
then let $\alpha_n \defas \alpha_{n-1}+1$, let
$X_n \defas \{x\^0 \st x \in X^2_n\}$
and let 
$p_n \in Q_{\alpha_n}$
be any non-redundant
quantifier-free $L_{\alpha_n}$-type with $|X^2_n|$-many free variables such that
\[
\models (\forall \x)\bigl(p_n(\x) \rightarrow p_{n-1}(\x)\bigr)
\]
where $|\x| = |X^2_n|$. We know that such a quantifier-free type exists by
condition (E). Then declare $p_n$ to be the quantifier-free
$L_{\alpha_n}$-type of $X_n$ (considered as an increasing ordered tuple under the
order induced from $X^2_n$) in $\X_n$. For $x\in X_n$, define $m_n(x)\defas
m^2_n(\pi(x))$, since every element has just one extension.

\noindent{\textbf{Case (b):}}
If, however, there is splitting of some order, i.e., condition (S) holds,
and the least such order
$\ell\in \Nats$ is no greater than $|X^2_n|$, then
let $\qn$ be some splitting of $p_{n-1}$ of order $\ell$. Let
$\alpha_n$ be the $e\in \Nats$ such that $\qn$ is a quantifier-free
$L_e$-type.

Define
\[
X_n \defas
\{x\^ 0\st x \in X^2_n\}
\cup
\{x\^ 1\st x \in X^2_n\}.
\]
and declare that $\qn$ is the quantifier-free $L_{\alpha_n}$-type of $X_n$
in $\X_n$, where $X_n$ is considered as the tuple
\[x_1\^0, x_1\^1, \ldots , x_{|X^2_n|}\^0, x_{|X^2_n|}\^1
\]
 where $x_1, \dots, x_{|X^2_n|}$ is increasing in the chosen order of $X^2_n$.

Finally, for each $x \in X_n$, define $m_n(x)  \defas m^2_n(\pi(x))/2$. In
other words, each element of $X^2_n$ has its mass divided evenly between
its two extensions. This concludes case (b).

\vspace{10pt}
Now, regardless of the case,  we verify property
($\mathscr{D}$) for stage $n$.
Suppose that
$g\in\Nats$ and  $s_1, \ldots, s_g \in X_n^2$ are distinct.
Further suppose that
$\ell_1, \ldots, \ell_g \in \Nats$,
such that
$s_1\^\ell_1, \ldots, s_g\^\ell_g \in X_{n}$.
Then note that for any quantifier-free $L_{\alpha_{n-1}}$-type $r$
with $g$-many free variables
\[ \X^2_n \models r(s_1, \ldots, s_g) \]
if and only if
\[\X_{n} \models r(s_1\^\ell_1, \ldots, s_g\^\ell_g).\]
Note that this property, composed with the analogous property verified at
the end of substage $n.2$, guarantees that
($\mathscr{D}$) holds.

Finally, for $n>0$ note that, by
property ($\mathscr{D}$),
if $\varphi_n\in \Lwow(L_{\alpha_{n-1}})$, then
for every
tuple $s_1, \ldots, s_{|a_{n}|} \in X_{n}$
with $\ov{a_{n}} \sqsubseteq \bigl(\pi^2(s_1), \ldots, \pi^2(s_{|a_{n}|})\bigr)$,
there is  an element $t\in X^1_n$
such that
\[\X^1_{n} \models \varphi_{n}\bigl(
\pi^2(s_1), \ldots, \pi^2(s_{|a_{n}|}) , t\bigr).\]
Hence if $t^* \in (\pi^2)^{-1}(t) \cap X_n$, then
\[\X_{n} \models \varphi_{n}(
s_1, \ldots, s_{|a_{n}|} , t^*).\]
This verifies property ($\mathscr{E}$).


\vspace{10pt}
\ul{Stage $\infty$:} Defining the Limiting Structure.

To complete the construction, we define the $L_\infty$-structure $\X_\infty$ via its
constantless part
$X_\infty$ and
the quanti\-fier-free $L_\infty$-type of every finite subset of $X_\infty$.
We also define the measure $m_\infty$.

Let
\[X_{\infty} \defas \{x\in \wtow \st (\forall i\in\Nats)\, (x|_{i} \in X_i)\},\]
and for each $n\in\Nats$ and each $y \in X_n$ define
\[m_\infty\bigl(\{x\in X_\infty \st x|_n = y\}\bigr) \defas  m_n(y).\]

Consider $X_\infty$ endowed with the topology inherited as a subspace of
$\wtow$ (itself under the product topology of $\Nats$ as a discrete set).
Then $X_\infty$ is the countable disjoint union $\bigcup_{\ell\in\Nats}
Y_\ell$,
where for each $\ell \in \Nats$,
\[
Y_\ell \defas
\{ \ell\^ a \st a \in \wtow \text{~and~} \ell \^ a \in X_\infty \}
\]
is a compact topological space having a basis of clopen sets,
under the topology inherited as a subspace of $\wtow$.
Hence $m_\infty$ can be extended in a unique way to a
countably additive measure on $X_\infty$.


For every $j, n \in\Nats$ and every
$s_1, \dots, s_j \in X_\infty$, there are some $n'\ge n$ and \linebreak
$t_1, \ldots t_j \in X_{n'}$ such that  for all distinct $i,i' \le j$,
\begin{itemize}
\item $t_i = s_i|_{2 n'}$ and
\item $t_i \ne t_{i'}$.
\end{itemize}
Let 
$q \in Q_{\alpha_n}$
be the quantifier-free $L_{\alpha_n}$-type such that
\[
\X_n \models q(t_1, \ldots, t_j).
\]
Then declare that
\[
\X_\infty\models q(s_1, \dots, s_j)
\]
holds.
This choice of quantifier-free type is well-defined because of property
($\mathscr{D}$) at all earlier stages.
This ends the construction.


\subsection{Invariant measures via the construction}

We now verify properties of $\X_\infty$ and $m_\infty$ that will allow us to produce the desired invariant measure.

\begin{proposition}
\label{atomless-nondegenerate}
The measure $m_\infty$ on
$X_\infty$
is a
non-degenerate atomless
probability measure.
\end{proposition}
\begin{proof}
The measures $m_n$ for $n\in \Nats$ cohere under projection
and agree with $m^*$, in the sense that 
\[
m_\infty (Y_n)
= m^*(n).
\]
But $m^*$ is a probability measure, and so
$m_\infty$ is as well.

For $n\in\Nats$ and $a\in X_n$, let
\[B_a \defas \{ s \in X_\infty \st s|_n = a \}.\]
The collection of sets of the form $B_a$
form a basis for the
topological space $X_\infty$.
Furthermore, for all $n\in\Nats$ and $a \in X_n$,
\[ m_\infty(B_a) = m_n(a) > 0.\]
Hence $m_\infty$ is non-degenerate.

For each $n\in\Nats$, define $\Gamma_n \defas \max \, \{m_n(a)\st a \in X_n\}$;
in substage $n.2$, we duplicate every element of $X_{n-1}$, and so
\[ \Gamma_{n} \le \Gamma_{n-1}/2.\]
Consider a singleton $\{ b \} \subseteq X_\infty$. Then
\[m_\infty(\{b\})
\le m_\infty(B_{b|_n})
\le \Gamma_n\] for each $n\in\Nats$, and so
$m_\infty(\{b\}) = 0$.
Hence $m_\infty$ is atomless.
\end{proof}

We will show that  $\X_\infty$ is an uncountable Borel model such that when
we sample countably infinitely many elements from $X_\infty$ 
independently according to the
probability measure $m_\infty$, the induced substructure
is almost surely a model of $T_\infty$.

\begin{proposition}
The structure $\X_\infty$ is a Borel $L_{\infty}$-structure.
\end{proposition}
\begin{proof}
Fix $n\in\Nats$.
Let $\psi$ be a quantifier-free $L_{\alpha_n}$-formula, and let $\ell$ be the
number of free variables of $\psi$. Then define the set of its
instantiating tuples:
\[
\Psi \defas
\{ a_1 \cdots a_\ell \in X_\infty \st
\X_\infty \models \psi(a_1, \ldots, a_\ell)
\}.
\]
Also define, for each $n'\ge n$,
\[
P_{n'}
\defas
\bigl\{ a_1 \cdots a_\ell \in X_\infty \st \X_{n'}\models \psi(a_1|_{2n'}, \ldots, a_\ell|_{2n'}) \bigr\}
\]
and
\[
I_{n'}
\defas
\bigl\{ a_1 \cdots a_\ell \in X_\infty \st
a_i = a_j \text{~iff~} a_i|_{2n'} = a_j|_{2n'},
\text{~whenever~} {1\le i \le j \le \ell}
\bigr\}
.
\]
Note that for each $n' \ge n$, both $P_{n'}$ and $I_{n'}$ are open sets.
We then have
\[
\Psi = \bigcup_{n'\ge n}  \bigl( P_{n'} \cap I_{n'}\bigr),
\]
and so $\Psi$ is an open set. As $\psi$ was arbitrary, $\X_\infty$ is a
Borel $L_\infty$-structure.
\end{proof}

A natural procedure for sampling 
substructures of $\X_\infty$
using $m_\infty$ will yield the desired
invariant measure.

Because $T_0$ is quantifier-free complete, all models of
$T_\infty$ have the same number of elements that instantiate constant
symbols, and the theory of equality between constants is fixed (as encoded
in $\cC_0$).

Let $\mu$ be an arbitrary atomless probability measure on $\X_\infty$.
We begin by describing a sampling procedure
that uses $\mu$
to determine an invariant measure $\mu^\circ$ on $\Models_{\cC_0, L_\infty}$:
First sample a countably infinite sequence of elements
$\<x_i\>_{i\in \Nats}$ from $X_\infty$ independently according to $\mu$.
If
there exist distinct $i, j\in \Nats$
such that $x_i = x_j$, then declare that all
atomic relations hold among all tuples; however, this occurs with
probability $0$, as $\mu$ is atomless.
Otherwise, for each quantifier-free
$L_\infty$-formula $\psi$, declare that
\[
\psi(n_1, \ldots, n_\ell)
\]
holds if and only if
\[
\X_\infty \models \psi(x_{n_1}, \ldots, x_{n_\ell})
\]
for all $n_1, \ldots, n_\ell \in \Nats$,
where $\ell$ is the number of free variables of $\psi$.
The
distribution of this random $L_\infty$-structure is a probability
measure on $\Models_{\cC_0, L_\infty}$; this is our desired $\mu^\circ$.
(As with the measures described via sampling in 
Section~\ref{simplified}, such probability measures are ergodic, as Kallenberg 
showed by extending the argument of Aldous \cite[Lemma~7.35]{MR2161313} to languages of unbounded arity in \cite[Lemma~7.22]{MR2161313} and \cite[Lemma~7.28 (iii)]{MR2161313}.)
Note that $\mu^\circ$ is
$\sym^{C_0}$-invariant, as $\<x_i\>_{i\in \Nats}$ is i.i.d.
Because $\mu$ is atomless, $\mu^\circ$ is concentrated on the class of structures with
underlying set $\Nats \cup C_0$ that are isomorphic to countably infinite 
substructures  of $\X_\infty$.

\begin{proposition}
\label{asModel}
The $\sym^{C_0}$-invariant probability measure $m_\infty^\circ$ on
$\Models_{\cC_0, L_\infty}$ is concentrated on the class of models of $T_\infty$.
\end{proposition}
\begin{proof}
By Proposition~\ref{atomless-nondegenerate},
the measure
$m_\infty$ is atomless, and so $m_\infty^\circ$ is an
$\sym^{C_0}$-invariant probability measure on
$\Models_{\cC_0, L_\infty}$ 
that is
concentrated on the class of countably infinite
substructures of $\X_\infty$.

Now let $\M$ be a sample from $m_\infty^\circ$, say via the
$m_\infty$-i.i.d.\ sequence $\<x_i\>_{i\in\Nats}$ of elements of $X_\infty$.
Fix an arbitrary $\eta \in T_\infty$. We will show that $\M \models \eta$
almost surely.  Because $\eta$ is pithy $\Pi_2$, we may write it in the
form $(\forall \z)(\exists y) \psi(\z, y)$ for some quantifier-free
$L_\infty$-formula  $\psi$. Let $\ell = |\z|$ and let $n\in\Nats$ be such
that $\psi \in L_{\alpha_n}$.
Fix an arbitrary tuple $\b \defas b_1 \cdots b_\ell \in \Nats$. We must show that there is
some $d\in \Nats$ such that
\[
\M \models \psi(\b, d) \quad \text{a.s.}
\]

Let $\overline{b^*}$ be the random tuple $x_{b_1}\cdots x_{b_\ell}$.
Let $j>n$ be any index of $\psi$ (i.e., such that $\psi = \varphi_j$)
satisfying $\overline{a_j}\sqsubseteq \overline{b^*}$. This is possible because of our
choice of repetitive enumeration.

By our construction in stage $j.2$, there is some  $e\in X^2_{j}$ such that
\[\X^2_{j}\models \psi(x_{b_1}|_{2j-2}  \cdots x_{b_\ell}|_{2j-2}, e) \quad
\text{a.s.}
\]
As in the proof of  Proposition~\ref{atomless-nondegenerate}, let
\[
B_e \defas \{ s \in X_\infty \st s|_{2j-2} = e \}.
\]
By our construction,  for any $e^* \in B_e$,
\[\X_{\infty}\models \psi(x_{b_1}  \cdots x_{b_\ell}, e^*)\quad \text{a.s.}
\]
However, $m_\infty(B_e)>0$,  and so
there is some $h\in\Nats$ such that
$x_h \in B_e \cap \M$, almost surely.
Hence
\[
\M \models \psi(\b, h) \quad \text{a.s.}
\]

Again by Proposition~\ref{atomless-nondegenerate}, 
the measure
$m_\infty$ is non-degenerate.
\end{proof}

We now show that if
the collection of quantifier-free types has splitting of
some order, 
the resulting construction
assigns measure $0$ to any particular isomorphism class of models of the
theory $T_\infty$.

\begin{theorem}
\label{maintheorem}
Suppose that $\<Q_i\>_{i\in \Nats}$ has splitting of some order. Then
there is an $\sym^{C_0}$-invariant probability measure on
$\Models_{\cC_0, L}$
that is concentrated on the class of models of
$T_\infty$ and is such that no single isomorphism class has positive measure.
\end{theorem}
\begin{proof}
Let $\ell\in \Nats$ be least such that $\<Q_i\>_{i\in \Nats}$ has splitting
of order $\ell$.
Let $m'$ be the $\sym^{C_0}$-invariant probability measure obtained in
Proposition~\ref{asModel}.
Define $M$ to be the
collection of isomorphism classes of countably infinite models of $T_\infty$ to which $m'$
assigns positive measure.

Suppose, to obtain a contradiction, that $M\neq \emptyset$. Then by the
countable additivity of $m'$, there can be at most countably many elements
of $M$. Hence among the quantifier-free $L_\infty$-types with
$\ell$-many free variables, at most countably many are realized in some
structure in $M$.
In particular, at most countably many \emph{non-constant} quantifier-free
$L_\infty$-types with
$\ell$-many free variables
are realized in some structure in $M$.
Then by countable additivity, there must be some non-constant quantifier-free
$L_\infty$-type $p$ with $\ell$-many free variables
that is realized in a positive fraction of models,
i.e., such that
\[
m'\bigl(\llrrC{(\exists \x)p(\x)}\bigr) >0,
\]
where $|x| = \ell$.

We then have
\[
0 < m'\bigl(\llrrC{(\exists \x)p(\x)}\bigr) =
m'\bigl(\bigcup_{\ttt \in \Nats^{\ell}}\llrrC{p(\ttt)}\bigr) \leq
\sum_{\ttt \in \Nats^\ell}m'\bigl(\llrrC{p(\ttt)}\bigr),
\]
where the equality is because $p$ is non-constant.
Hence there is some $\ttt\in \Nats^\ell$ such that
$m'\bigl(\llrrC{p(\ttt)}\bigr) > 0$,
by the countable additivity of $m'$.

For every $i \in \Nats \cup \{\infty\}$, define 
$\eta_i \defas
m'\bigl(\llrrC{p|_{L_{\alpha_i}}(\ttt)}\bigr)$.
Because \[L_0 \subseteq L_1 \subseteq \cdots \subseteq L_\infty,\] we have
$\eta_i\geq \eta_j$ whenever $0 \leq i < j \leq \infty$.

Let $g \ge \ell$ be arbitrary. We will show that
\[\eta_{g} \le 2^{-g} +  (1 - 2^{-\ell})^{g - \ell}.\]
This will imply that $\eta_\infty \le \inf_i (1 - 2^{-\ell})^{2 i} = 0$, and
so $m'\bigl(\llrrC{p(\ttt)}\bigr)  = 0$, a contradiction.

There are two (overlapping) ways 
that
an $\ell$-tuple 
of elements of
$X_\infty$ 
sampled independently
according to $m'$
can fail to satisfy $p|_{L_{\alpha_g}}$: either (1) the
restriction of the tuple to $\Nats^{2g}$ satisfies a redundant quantifier-free type, 
in which case
the tuple might not satisfy $p|_{L_{\alpha_g}}$, or (2) its
restriction to $\Nats^{2g}$ is non-redundant but satisfies some
quantifier-free type other than $p|_{L_{\alpha_g}}$.

By our choice of $\Lambda_g$ in stage $g.2$, we know that for any
assignment of mass to $X^1_g$,  the probability of an independently
selected $\ell$-tuple having two elements selected from the same element of
$X^2_g$ is no more than $2^{-g}$, as $g \ge \ell$. Hence 
the probability that
(1) 
occurs
is
bounded by $2^{-g}$.

Because the mass of every element is split evenly between those elements
descending from it via iterated duplication,
the probability that a given non-redundant $\ell$-tuple of $X^2_g$ is selected
independently according to $m^2_g$  is $2^\ell$ times the probability that
any of such duplicated elements are
selected independently according to
$m_g$.

Let $\zeta_g$ be the probability that a given $\ell$-tuple, independently
selected from $X_g$ \nobreak according to $m_g$, has quantifier-free type
$p|_{L_{\alpha_g}}$ conditioned on the fact each element of the
$\ell$-tuple is distinct (i.e., $\zeta_g$ is a bound on 
the probability that
(2)
occurs, so that
$\eta_g \le
2^{-g} + \zeta_g$).
By 
the
splitting of quantifier-free types in stage $g.3$, we know that
for every $\ell$-tuple in $X^2_g$ there are at least two quantifier-free
$L_{\alpha_{g}}$-types of duplicates of the $\ell$-tuple.

Hence we have
\[\zeta_g \le (1 - 2^{-\ell})\cdot \zeta_{g - 1}
\le (1 - 2^{-\ell})^{g - \ell} .
\]

In total, we have $\eta_g \le 2^{-g} +  (1 - 2^{-\ell})^{g - \ell}$.
\end{proof}


\section{Approximately
$\aleph_0$-categorical theories}
\label{almostSec}

In this section,
we introduce several conditions
on
first-order theories that 
together
allow us to apply
Theorem~\ref{maintheorem}.
These
will give us
an invariant probability measure that is concentrated on
the class of
models of 
a theory, but
does not 
assign
positive measure to any single
isomorphism class of 
models.
We then give
examples of first-order theories satisfying these conditions.

Key among these conditions is a property that we call
\emph{approximate
$\aleph_0$-categoricity}.

\begin{definition}
Let $L$ be a countable language.
A first-order theory $T \subseteq \Lww(L)$ is \defn{approximately
$\aleph_0$-categorical} when
there is a sequence of languages $\<L_i\>_{i\in\Nats}$, called a
\defn{witnessing sequence}, such that
\begin{itemize}
\item $L_i \subseteq L_{i+1}$ for all $i \in\Nats$,
\item $L = \bigcup_{i \in \Nats}L_i$, and
\item $T\cap \Lww(L_i)$ is $\aleph_0$-categorical for each $i \in \Nats$.
\end{itemize}
\end{definition}

In particular,
any 
approximately
$\aleph_0$-categorical
theory is
the countable union of
$\aleph_0$-categorical first-order theories (in different languages).

We now give criteria under which the class of models of an 
approximately \linebreak
$\aleph_0$-categorical theory admits an invariant probability measure
that assigns measure 
$0$
to any single isomorphism class of models.

Recall the notion of a pithy $\Pi_2$ expansion from \S\ref{defExpSec}.
Note that any model of a first-order $L$-theory $T$ has a unique expansion to a model of its pithy $\Pi_2$ expansion. Furthermore, any invariant measure concentrated on a Borel set $X \subseteq \Models_L$ can be expanded uniquely to an invariant measure concentrated on 
\[
\{ \M^* \in \Models_{L_\HF} \st \M^*|_L \in X\}.
\]

\begin{lemma}
\label{Tstar-pithy}
Let $L$ be a countable 
language, and suppose that $T$ is an 
approximately
$\aleph_0$-categorical
$\Lww(L)$-theory with witnessing sequence $\<L_i\>_{i\in\Nats}$.
Then the pithy $\Pi_2$ expansion $T^*$ of $T$ is also 
approximately
$\aleph_0$-categorical.
\end{lemma}
\begin{proof}
For each $i\in\Nats$, 
the $L_i$-theory $T \cap \Lww(L_i)$ is $\aleph_0$-categorical by
hypothesis. For each $i$, let
$L_i^*$ be the language of the pithy $\Pi_2$
expansion $T_i^*$ of $T \cap \Lww(L_i)$. Then
each $T_i^*$ is
$\aleph_0$-categorical. 
Note that $T^* \cap \Lww(L_i^*)  = T_i^*$ for each $i \in \Nats$, and
$\<L_i^*\>_{i\in\Nats}$ is a nested sequence whose union is the language of
$T^*$.
Hence $T^*$ is 
approximately
$\aleph_0$-categorical
with witnessing sequence
$\<L_i^*\>_{i\in\Nats}$.
\end{proof}

The following result is now straightforward from 
Theorem~\ref{maintheorem}.

\begin{theorem}
\label{SplittingTypesTheorem}
Let $L$ be a countable 
relational
language, and suppose that $T$ is an 
approximately
$\aleph_0$-categorical
$\Lww(L)$-theory
with witnessing sequence $\<L_i\>_{i\in\Nats}$.
For each $i\in\Nats$, let $Q_i$ be 
any enumeration of the
quantifier-free $L_i$-types that
are consistent with $T\cap\Lww(L_i)$.
Further suppose that
\begin{itemize}
\item for each $i \in \Nats$, the age of the unique countable model
(up to
isomorphism) of $T\cap \Lww(L_i)$ has
the strong amalgamation property,
and
\item
the sequence $\<Q_i\>_{i\in\Nats}$ has splitting of some order.
\end{itemize}
Then there is an $\sym$-invariant probability measure on $\Models_L$ that is concentrated
on the class of models of $T$ but that assigns measure $0$ to each isomorphism class of
models.
\end{theorem}
\begin{proof}
By Lemma~\ref{Tstar-pithy}, the pithy $\Pi_2$ expansion $T^*$ of $T$ is approximately $\aleph_0$-categorical.
Note that for each $i\in\Nats$, every element of $Q_i$ is consistent with the pithy $\Pi_2$ expansion of $T \cap \Lww(L_i)$.
We may therefore 
run the construction of \S\ref{construction-subsec},
under the assumption that conditions (W), (D), (E), and (C)
hold
of $\<Q_i\>_{i\in\Nats}$. 
Under the further assumption that (S) holds of
$\<Q_i\>_{i\in\Nats}$,
we may
apply
Theorem~\ref{maintheorem}
to
obtain
an invariant measure on $\Models_{L_\HF}$ that is concentrated on the class of models of $T^*$ but that assigns measure $0$ to each isomorphism class.
The restriction of this invariant measure to $\Models_L$ will give us an invariant measure with the desired properties.
We now show that 
these five conditions
hold of $\<Q_i\>_{i\in\Nats}$.

Condition (D) follows from our first hypothesis, and (S) from our second.

Conditions (E) and (C) hold of $\<Q_i\>_{i\in\Nats}$ because
for each $i\in\Nats$, the set $Q_i$  contains
\emph{every} quantifier-free $L_i$-type that
is consistent with $T\cap\Lww(L_i)$.

Finally, we show condition (W). Note that
any pithy $\Pi_2$ sentence
\[(\forall \x)(\exists y)\psi(\x,y)\in T
\] is an $L_n$-formula for some $n\in\Nats$.
Hence as $Q_n$ is consistent with $T \cap \Lww(L_n)$,
for any quantifier-free $L_n$-type $q\in Q_n$,
there is some $q'\in Q_n$ extending $q$ 
such that 
for every tuple $\z$ of free variables of $q$ having size $|\x|$,
\[
\models
(\forall \ww)\bigl(
q'(\ww) \to
(\exists y) 
\psi(\z,y)\bigr)
\]
holds, where $|\ww|$ is the number of free variables of $q'$.
Therefore
condition (W) holds of  $\<Q_i\>_{i\in\Nats}$.
\end{proof}

In particular, a theory satisfying the hypotheses of 
Theorem~\ref{SplittingTypesTheorem} is not itself \linebreak
$\aleph_0$-categorical, as it must have uncountably many countable models.
We now use this theorem to give examples of an invariant measure
that is concentrated on the class of models of a first-order theory but 
but that assigns
measure
$0$ to each isomorphism class of models.


\subsection{Kaleidoscope theories}
Here
we show a simple way in which 
a
\Fr\ limit whose age has the strong amalgamation property
gives rise to
an 
approximately
\linebreak
$\aleph_0$-categorical theory,
which we call
its 
corresponding 
\emph{Kaleidoscope} theory,
whose countable models consist of
countably infinitely many copies of the \Fr\ limit combined in an appropriate way.
Furthermore, we show that if 
such a
\Fr\ limit 
satisfies the mild condition that
for some finite size its age has 
at least two non-equal
structures of that size (not necessarily non-isomorphic),  then
its Kaleidoscope theory satisfies the hypotheses of 
Theorem~\ref{SplittingTypesTheorem}.

\begin{definition}
Suppose $L$ is a countable relational language. 
Let 
$\<L^j\>_{j\in\Nats}$
be an infinite sequence of pairwise
disjoint copies of $L$ such that $L^0 = L$,  and
for $i\in\Nats$, define $L_i \defas \bigcup_{0\le j\le i} L^j$.
\end{definition}

\begin{lemma}
\label{AiAge}
Let $L$ be a 
countable
relational
language, and let $A$ be a 
strong amalgamation
class 
of $L$-structures.
For each $i\in\Nats$,
define $A_i$ to be the 
class
of all finite $L_i$-structures  $\M$ such that for
$0\le j\le i$, the reduct $\M|_{L^j}$ 
(when considered as an $L$-structure)
is in $A$.
Then each
$A_i$ is a 
strong amalgamation class.
\end{lemma}
\begin{proof}
Each $A_i$ satisfies the strong amalgamation property: Suppose
$\M$, $\N \in A_i$ have a common substructure $\O\in A_i$. 
For each $j$ such that $0 \le j \le i$, let $\X^j$ be a strong amalgam of
$\M|_{L^j}$ and $\N|_{L^j}$ over $\O|_{L^j}$. Because $\X^0, \ldots, \X^i$
are in disjoint languages and have the same underlying set, there is an
$L_i$-structure $\X$ on this underlying set such that for $0 \le j \le
i$, we have $\X|_{L^j} = \X^j$. Hence $\X\in A_i$ is a strong amalgam of $\M$, $\N$
over $\O$. 

Each $A_i$ is a 
class containing countably many isomorphism types,
for which the
hereditary property holds trivially.
Further, the joint embedding property holds by a similar argument to that above.
Thus each $A_i$ is a strong amalgamation
class.
\end{proof}

\begin{definition}
Using the 
notation of Lemma~\ref{AiAge},
for each $i\in \Nats$, let $T_i$  be the theory of 
the \Fr\ limit of $A_i$, and notice that $T_i \subseteq T_{i+1}$.
The theory $T_\infty \defas \bigcup_{i\in\Nats} T_i$
in the language $L_\infty \defas \bigcup_{i\in\Nats} L_i =
\bigcup_{j\in\Nats} L^j$
is therefore consistent.
The theory $T_\infty$ is said to be the \defn{Kaleidoscope theory}
built from $A$.
\end{definition}

\begin{proposition}
\label{KaleidoscopeInvMeas}
Let $L$ be a countable relational language, and let $A$ be a 
strong amalgamation class
of $L$-structures.
Let $T_\infty$, in the language $L_\infty$, be the 
Kaleidoscope  
theory built from $A$, as above.
Then $T_\infty$ is 
approximately
$\aleph_0$-categorical.

Furthermore, suppose that for some $n\in\Nats$, the age $A$ has at least two
non-equal elements of size $n$ on the same underlying set.
(Note that we do not require these elements to be non-isomorphic.)
Then
there is an $\sym$-invariant probability measure on $\Models_{L\infty}$ that is concentrated
on the class of models of $T_\infty$ but that assigns measure $0$ to each isomorphism class of
models.
\end{proposition}
\begin{proof}
For each $i\in\Nats$, let $A_i$ be 
as defined in 
Lemma~\ref{AiAge}; then
$A_i$ is 
the age of a model of $T_i$, which
is
an $\aleph_0$-categorical $L_i$-theory.
Therefore $T_\infty$
is an 
approximately
$\aleph_0$-categorical $L_\infty$-theory with
witnessing 
sequence $\<L_i\>_{i\in\Nats}$.

We will apply Theorem~\ref{SplittingTypesTheorem} to obtain the desired
invariant measure. We must show its two hypotheses: the strong amalgamation
property
for
the age of each $T_\infty \cap \Lww(L_i)$, and that
$\<Q_i\>_{i\in\Nats}$
(as defined in Theorem~\ref{SplittingTypesTheorem})
has splitting of some order.

For any $i\in\Nats$,
because 
$A_i$
is the age of the unique model of $T_i = T_\infty \cap \Lww(L_i)$,
we may apply
Lemma~\ref{AiAge} to see that 
$A_i$ 
is a
strong amalgamation 
class
as well.

We now show that $\<Q_i\>_{i\in\Nats}$
has splitting of order $n$. Fix $j\in\Nats$, and
let $q \in Q_j$ be a non-redundant quantifier-free $L_j$-type with
$k$-many free
variables,
for some $k>n$.
It suffices to find, for some $j'>j$, a quantifier-free type
$\qn\in Q_{j'}$ 
with free variables
$\xx \defas x_1^0, x_1^1, \ldots, x_k^0, x_k^1$
such that the
restriction $\qn$ to $L_j$
is an iterated duplicate of $q$, 
and
for any
$2n$-tuple $y_1 \cdots y_n z_1 \cdots z_n$ of distinct free variables of
$\qn$,
we have
\[
\qn|_{y_1, \ldots, y_n} \neq
\qn|_{z_1, \ldots, z_n},
\]
which ensures that $\qn$ is a splitting of $q$.
We construct $\qn$ in the following manner.

In languages $L^0, \ldots, L^j$, the quantifier-free type $\qn$ 
describes
an iterated duplicate of $q$; each of the remaining languages $L^{j+1}, \ldots,
L^{j'}$, corresponds to a particular way of choosing a $2n$-tuple of
variables from the 
$2k$-tuple $\xx$, and describes a pair of different $n$-element structures on
this $2n$-tuple.
Let $q^*$ be the quantifier-free $L_j$-type with free variables
$\xx$
that is an iterated duplicate of $q$.
Let $B_0$ and $B_1$ be two non-equal elements of $A$ of size $n$ on the
same underlying set 
$\{0, \ldots, n-1\}$,
and let $p_0, p_1 \in Q_0$
be
quantifier-free $L$-types 
such that
\[
B_i \models p_i(0, \ldots, n-1)
\]
for $i\in\{0,1\}$. 
By the joint embedding property of $A$, let
$p\in Q_0$
be any quantifier-free $L$-type 
with
$2n$-many free variables $v_1, \ldots, v_n, w_1, \ldots, w_n$ such that
\[ p(\vv, \ww)  \to p_0(\vv) \wedge p_1(\ww),
\]
where $\vv \defas  v_1\cdots v_n$ and $\ww \defas w_1 \cdots w_n$.

Enumerate all
$2n$-tuples of distinct variables of $\xx$.
Assign 
each
such tuple $\uu$ 
a distinct
value 
\[j_\uu \in [j+1, \ldots, j'],\]
where $j' \defas j+ (2k)(2k-1)\cdots(2k - 2n+1)$.
For each such tuple $\uu$,
choose a 
quantifier-free $L$-type $q_\uu$ with
free variables
$\xx$
such that
\[
\models (\forall  \xx)
\bigl(q_\uu(
\xx
) \to p(\uu)\bigr).
\]
Let $\qn$ be  a quantifier-free $L_{j'}$-type
with free variables
$\xx$
that implies  $q^*(\xx)$ and  that also implies, for each such tuple $\uu$,
that $q^{L^{j_\uu}}_\uu(\xx)$ holds, where $q_\uu^{L^{j_\uu}}$ describes in language
$L^{j_\uu}$ what $q_\uu$ describes in $L$.
Note that we can find such a $\qn$ because the restrictions of $T_\infty$
to each copy of $L$ do not interact with each other.
Finally, because $p(\vv, \ww) \neq p(\ww, \vv)$,  for any
$2n$-tuple $y_1 \cdots y_n z_1 \cdots z_n$ of distinct free variables of
$\qn$, we have that
\[
\qn|_{y_1, \ldots, y_n} \neq
\qn|_{z_1, \ldots, z_n} .
\]

Therefore
$\<Q_i\>_{i\in\Nats}$
has splitting of order $n$.
\end{proof}

A key example of this construction
is provided by what we call the
\emph{Kaleidoscope random graphs}, which  
are the countable models of the Kaleidoscope theory 
built from the class of
finite graphs (in the language of graphs).
There are
continuum-many
Kaleidoscope random graphs (up to isomorphism).
Each 
Kaleidoscope random graph $G$
can be thought of as
countably many random graphs (i.e., Rado graphs), each with a different
color for its edge-set, overlaid on the same
vertex-set
in such a way 
that
for 
every finite 
substructure  $F$ of $G$ and any chosen finite set of colors,
there is 
an extension of $F$ by a single vertex $v$ of $G$ satisfying
any
given assignment of edges and non-edges in those colors between 
$v$
and the 
vertices of $F$.

The invariant measures provided by
Proposition~\ref{KaleidoscopeInvMeas} 
are
fundamentally different from those obtained
in \cite{AFP}. 
No
measure provided by Proposition~\ref{KaleidoscopeInvMeas} is
concentrated on the isomorphism class of a single structure, nor is any such measure
concentrated on a class of structures 
having 
trivial definable closure. To see this, consider
such a measure,
and suppose $n\in\Nats$ is such that the age $A$ has at least two elements
of size $n$. Then for a structure sampled from the invariant measure, with
probability $1$ the 
tuple
$0, \ldots, n-1$
has
a quantifier-free type
different from that of every other $n$-tuple in the structure.
Hence the structures sampled from such a measure almost surely do not have
trivial definable closure. As a consequence of this and
the main result of \cite{AFP},
for almost 
every structure
sampled
from this measure,
there 
is no
invariant measure concentrated on the isomorphism class of just that structure.


\subsection{Urysohn space}
The \emph{Urysohn space} $\UU$ is the universal ultrahomogeneous
Polish space.
In other words, up to isomorphism (i.e., bijective isometry), $\UU$ is the
unique 
complete separable metric space 
that is \emph{universal},
in 
that $\UU$ contains an 
isomorphic
copy of every complete separable metric space,
and 
\emph{ultrahomogeneous},
in 
that every 
isomorphism
between two finite subsets of $\UU$ can be extended to an 
isomorphism
of the entire space $\UU$.

Although Urysohn's work predates that of \Fr~\cite{MR0057220},
his construction of $\UU$ can be viewed as a continuous generalization of the \Fr\ method.
Hu\v{s}ek \cite{MR2435145}
describes Urysohn's original construction
\cite{Urysohn27}
and its history, and Kat\v{e}tov's more recent generalizations
\cite{MR952617}.
For further background, see the introductory remarks in
Hubi\v{c}ka--Ne\v{s}et\v{r}il \cite{MR2435144} and
Cameron--Vershik \cite{MR2258622}.
For perspectives from model theory and descriptive set theory, see,
e.g., 
Ealy--Goldbring \cite{MR2951630},
Melleray \cite{MR2435148},
Pestov \cite{MR2435149}, and
Usvyatsov \cite{MR2435152}.

Vershik \cite{MR2006015}, \cite{MR2086637} has demonstrated how
Urysohn space,
in addition to being the universal ultrahomogeneous Polish space,
also can be viewed as the \emph{generic} Polish space, and as a
\emph{random} Polish space.
Namely, 
Vershik shows that
$\UU$ is the generic complete separable metric space, in the sense of Baire category,
and he provides symmetric random constructions of $\UU$ by describing a wide class of
invariant measures  concentrated on the class of metric spaces whose completion is
$\UU$. As with the constructions in
\cite{MR2724668} and \cite{AFP},
these measures are determined by
sampling from certain continuum-sized structures.

Here we
construct an
approximately
$\aleph_0$-categorical theory
whose models 
are
those countable metric spaces 
(encoded in an infinite relational language)
that have
Urysohn
space
as their completion.
Hence
our invariant probability measure concentrated on the class of models of
this theory can be thought of as providing yet another symmetric random construction of Urysohn space.

Before describing the theory itself, we 
provide a relational
axiomatization of metric spaces
using infinitely many binary relations,
where the distance function is implicit  in these
relations.
Let 
$\LMS$ be the language consisting of a binary relation $d_{q}$ 
for
every $q\in\RatNonneg$. Given a metric space with distance function
$\dd$, the
intended interpretation will be that $d_q(x,y)$ holds when
$\dd(x,y) \le q$.
More explicitly, we have, for all $q, r \in \RatNonneg$,
\begin{itemize}
\item $(\forall x)(\forall y) \ \bigl( d_q(x,y) \rightarrow d_r(x,y)\bigr)$ when
$r \ge q$,
\item $(\forall x)(\forall y) \ \bigl(d_q(x,y) \leftrightarrow d_q(y,x)\bigr)$,
\item $ (\forall x) (\forall y) (\forall z) \ \bigl( (d_q(x,y) \And d_r(y,z)) \rightarrow
d_{q+r}(x, z) \bigr)$, and
\item 
$(\forall x) \ d_0(x,x).$
\end{itemize}
Let $\TMS$ denote this theory in the language $\LMS$.

The following 
result
is immediate.
\begin{proposition}
\label{metricIff}
For every metric space $\SS = (S, \dd_S)$, the $\LMS$-structure $\M_\SS$ with underlying set $S$
and sequence of relations
$\<d_q^{\M_\SS}\>_{q\in\RatNonneg}$ defined by
\[d^{\M_\SS}_q(x,y) 
\quad \text{if and only if} \quad
\dd_S(x, y) \le q\]
is a model of  $\TMS$.

Conversely, if $\N$ is a model of $\TMS$ with
underlying set $N$, and
\[\dd_\N(x,y) \defas \inf\, \{ q \in \RatNonneg \st \N \models d_q(x,y)\},\]
then $\W_\N \defas (N, \dd_\N)$ is a metric space.
\end{proposition}

We will use the maps $\SS \mapsto \M_\SS$ and $\N \mapsto \W_\N$ that are
implicit in
Proposition~\ref{metricIff} throughout our discussion of Urysohn space.

Note that when a model $\N$ of $\TMS$ further satisfies, for each $q\in
\RatNonneg$,
the infinitary axioms
\begin{itemize}
\item $\displaystyle 
(\forall x) \ \Bigl (\bigl(\bigwedge_{p> q} d_p(x,
y)\bigr) \rightarrow d_q(x,y) \Bigr)$ and
\item $\displaystyle (\forall x) (\forall y) \ \bigl(d_0(x,y) \rightarrow (x = y)\bigr),$
\end{itemize}
then $\N = \M_\SS$  for some metric space $\SS$.
However, we will not be
able to ensure 
that these axioms hold
in 
our construction,
each stage of which involves a language that has
only
a finite number of relations of
the form $d_q$.

\begin{proposition}
For any finite sublanguage $L$ of $\LMS$, every model of
the restriction 
$\TMS\cap \Lww(L)$
of $\TMS$ can be extended to a model of $\TMS$.
\label{conservative}
\end{proposition}
\begin{proof}
Let $L$ be a finite sublanguage of $\LMS$, and let $\N$ be a model of
$\TMS\cap \Lww(L)$
with underlying set $N$.
Define
\[
\Rationals_L \defas \{ q \in \RatNonneg \st d_q\in L\}.
\]
Let $p \defas \max\, \Rationals_L$.
For 
every pair of distinct elements $x,y\in \N$, define
\[\delta^*_\N(x,y) \defas \min\bigl(2p, \ 
\inf\, \{ q \in \Rationals_L \st \N \models d_q(x,y)\} \bigr),
\]
and for all $x\in \N$ set
\[\delta^*_\N(x,x) \defas 0.
\]
Finally, define
\[\delta_\N(x,y) \defas \inf\, \{ \delta^*_\N(x, z_1) + \delta^*_\N(z_1, z_2) + \cdots + \delta^*_\N(z_n, y) \st 
n\ge 1\text{~and~} z_1, \ldots, z_n \in \N 
\}.
\]
Although
$(N,\delta_\N)$ need not be a metric space,
the $\LMS$-structure
$\M_{(N,\delta_\N)}$, given by the map defined in Proposition~\ref{metricIff},
is a model of $\TMS$.
By construction, if \linebreak $\N \models d_q(x,y)$, then $\delta_\N(x,y) \le q$.
However, if $\N \models \neg d_q(x,y)$, then by the triangle inequality $\delta_\N(x,y) > q$.
Hence $(N,\delta_\N)$ is 
consistent with the above ``intended interpretation'' of the relations in $\N$.
In particular, $\M_{(N,\delta_\N)}$ is an expansion of $\N$ to $\LMS$ 
that is a model of $\TMS$.
\end{proof}

We 
now describe 
an important class of
examples of countable metric spaces whose
completions are (isomorphic to)
the full Urysohn space.

\begin{definition}
Let
$D$
be a 
countable dense subset of
$\Rplus$.
Consider the class 
$\mathscr{S}$
of \linebreak finite metric spaces $\SS$ whose
non-zero distances 
occur
in $D$, and let \linebreak $\mathscr{F} \defas \{\M_\SS \st \SS \in \mathscr{S}\}$. Note that
$\mathscr{F}$ is an amalgamation 
class.
Define $D\UU$ to be $\W_\N$, where $\N$ is
the \Fr\ limit of 
$\mathscr{F}$.
\end{definition}

It is a standard result
that any such $D\UU$ is a metric space whose completion is $\UU$.
The particular case $\Rationals \UU$ has been well-studied, and is
known as the \emph{rational Urysohn space}.

We now extend $\TMS$ to an $\LMS$-theory $T_U$
whose countable models will be precisely those
$\LMS$-structures $\N$ 
for which
the completion of $\W_\N$ is 
isomorphic
to $\UU$.
We will work with 
finite sublanguages of $\LMS$,
rather than all of $\LMS$,
because there is no (countable) \Fr\ limit of the class of finite models of $\TMS$; in particular,
there are continuum-many non-isomorphic finite models of $\TMS$, even of size $2$. On the other hand,
in every finite sublanguage $L$ of $\LMS$, there is a \Fr\ limit of
the countably many (up to isomorphism) finite models of 
$\TMS\cap \Lww(L)$.

\begin{definition}
Let $L$ be a finite sublanguage of
$\LMS$. Note that the class of finite models of 
$\TMS\cap \Lww(L)$ 
is an amalgamation 
class.
Let $T_U^L$ be the $\Lww(L)$-theory 
of the \Fr\ limit of this class, and
define
\[T_U \defas
\bigcup\bigl\{T_U^L \st \text{finite~} L\subseteq \LMS\bigr\}.\]
\end{definition}

\begin{proposition}
The theory $T_{U}$ is consistent.
\label{consistentTu}
\end{proposition}
\begin{proof}
Consider the $\LMS$-structure
$\M_{\Rationals\UU}$.
It is a \Fr\ limit of 
the class of those finite models $\N$ of $\TMS$ 
for which
$\W_\N$ is a metric space with only rational distances.
By Proposition~\ref{conservative}, 
and as
$\Rationals$ is dense in $\Reals$,
for any finite sublanguage $L$ of $\LMS$,
the \Fr\ limit of the class of finite
models of 
$\TMS\cap \Lww(L)$ 
is isomorphic to 
$\M_{\Rationals\UU}|_L$.
Hence
$\M_{\Rationals\UU}|_L$ is a model of $T_U^L$. Therefore
$\M_{\Rationals\UU}$
is a model of $T_U$,  and so
$T_U$ is consistent.
\end{proof}

Note that by the above proof, for any countable dense subset
$D\subseteq\Rplus$, the \linebreak $\LMS$-structure $D\UU$ is a model of $T_U$. As
these are all non-isomorphic, $T_U$ has continuum-many countable models.
Also note that for any finite sublanguage $L$ of $\LMS$ 
and dense $D,
E\subseteq\Rplus$, the $L$-structures $\M_{D\UU}|_L$ and $\M_{E\UU}|_L$ are
isomorphic (and are both \Fr\ limits as in the above proof).

\begin{theorem}
\label{otherdirection}
Let $\SS = (S, \dd_S)$ be a countable metric space.
Then $\M_\SS$ is a model of $T_U$ if and only if
the completion of $\SS$ is 
isomorphic
to $\UU$.
\end{theorem}
\begin{proof}
First suppose that the completion of $\SS$ is 
isomorphic
to  $\UU$.
Without loss of generality, we may assume that $\SS \subseteq \UU$ and
that $\SS$ is dense in $\UU$.
We will show that $\M_\SS$ is a model of $T_U$.

Let $L$ be any finite sublanguage of $\LMS$,
and suppose 
that
\[(\forall \x)(\exists y)\varphi(\x,y) \in T_U\cap \Lww(L).\]
Because $T_U\cap \Lww(L)$ has a pithy $\Pi_2$ axiomatization, it
suffices to
show that 
\linebreak
$(\forall \x)(\exists y)\varphi(\x,y)$
holds in $\M_\SS$.

Fix some $\a \in \M_\SS$ where $|\a|$ is one less than the number of
free variables of $\varphi$, and let $q$ be the quantifier-free $L$-type
of $\a$. We will show that there is 
a
witness to
$(\exists y)\varphi(\a, y)$ in $\M_\SS$.

Because $T_U$ implies the theory of the \Fr\ limit of the class of
finite $L$-structures,
there is some quantifier-free $L$-type $q'(\x, y)$ extending
$q(\x)$ (where $|\x|= |\a|$)
that is
consistent with both $\varphi(\x, y)$ and 
$\TMS\cap \Lww(L)$.

Now, $\UU$ is universal for separable metric spaces, and so there is
some tuple $\c f \in \UU$ such that $q'$ holds of $\M_\CC$ 
(under
the
corresponding order of elements), where $\CC$ is the
substructure of $\UU$ with underlying set $\c f$. As $\UU$ is
ultrahomogeneous and $q$ is the quantifier-free type of $\a$, there must be an automorphism $\sigma$ of $\UU$ such that
$\sigma(\c) = \a$. Define $b \defas \sigma(f)$.
Then
$q'$ holds of $\M_\BB$ (in the corresponding order),
where $\BB$ is the substructure of $\UU$ with underlying set $\a b$.

But no quantifier-free $L$-type can ever
completely determine the distance 
between any two distinct points, as $L$ is finite. Hence
there is
some $\varepsilon>0$ such that 
$q'$ also holds of $\M_\AA$ (in the corresponding order)
whenever
$\AA$ is any finite $(|\a|+1)$-element 
substructure
of $\UU$ that can be put into one-to-one correspondence with $\a b$
in such a way that each element of $\AA$ is
less than $\varepsilon$ away from 
the corresponding
element of $\a b$ and from no other.
By assumption, $\SS$ is dense in $\UU$, 
and so there is some $b'\in \SS$ such that $d_\UU(b, b') < \varepsilon$.
Hence $\M_\SS \models q'(\a, b')$, and so $\M_\SS \models \varphi(\a, b')$,
as desired. 

Conversely, suppose that $\SS$ is a countable metric space such that
$\M_\SS$ is a model of $T_U$. We will show that the completion $\cU$ of
$\SS$ is 
isomorphic
to  $\UU$.

We do this by showing
that for every
finite metric space $\AA$ with underlying set $A \subseteq \cU$ and
metric space $\BB$ extending $\AA$ by some element 
$b$ (not necessarily in $\cU$),
there is some $b'\in \cU$ such that the metric space induced (in $\cU$)
by $A\cup
\{b'\}$ is 
isomorphic
to $\BB$.
From this it follows that if $\sigma$ is an isomorphism from $\AA$ to
another submetric space $\AA'$ of $\cU$, then for every $c\in\cU$, 
there is some
$c'\in\cU$ such that the function that extends $\sigma$ 
by mapping $c$ to $c'$ is also an isomorphism of induced metric spaces.
By a standard back-and-forth argument, this 
implies
the universality and ultrahomogeneity of $\cU$. Hence
$\cU$ is 
isomorphic
to
$\UU$,
as $\UU$ is the unique (up to 
isomorphism) 
universal ultrahomogeneous 
complete separable
metric space.

Let $\AA$ and $\BB$ be as above, and suppose $A = \{a_0, \ldots, a_{n-1}\}$,
where $n = |A|$.
Let $\cU^*$ be any metric space extending $\cU$ by $b$.
and define 
\[
\gamma_j \defas \dd_{\cU^*}(a_j, b)
\] for $0 \le j < n$.
Let $\<L_i\>_{i\in\Nats}$
be an increasing sequence
of finite sublanguages of $\LMS$ 
such that  
for each $i\in\Nats$, 
the language $L_i$ contains enough symbols of the form $d_r$ to imply that
whenever two finite models of $\TMS$, both of
diameter less than twice that of $\BB$, satisfy the same
quantifier-free $L_i$-type (in some order), then each pairwise
distance in the first structure is within $2^{-(i+6)}$ of the corresponding
distance in the second structure.
For each $j$  such that $0\le j\le n-1$, let $\<a_j^i\>_{i\in \Nats}$ be a
Cauchy sequence in $\SS$ that converges to $a_j$ with
\[
\dd_\SS(a_j^i, a_j^{i+1}) \le 2^{-(i+3)}
\]
for $i\in\Nats$.

Consider the inductive claim
that for $h\in\Nats$ we have defined $b_0
\cdots b_h\in \SS$ that satisfy
\[\dd_{\SS}(b_i, b_{i+1}) \le 2^{-i}\] 
for 
$i < h$, and 
\[
\bigl|\dd_\SS(a_j^{i}, b_{i})
- \gamma_j
\bigr| \le 2^{-(i+2)},
\]
for $0 \le j\le n-1$ and
$i \le h$.

If this claim holds for all $h\in\Nats$, then
$\<b_i\>_{i\in\Nats}$ is a Cauchy sequence in $\SS$, which
therefore must converge
to an element $b' \in \cU$. 
Furthermore, 
$
\dd_{\cU}(a_j, b') = \gamma_j
$
for \linebreak $0 \le j\le n-1$,
and so
the metric space induced by $A\cup \{b'\}$ is 
isomorphic
to $\BB$,
as desired.

We now show the inductive claim for $h+1$. Because 
\[
\bigl|\dd_{\cU^*}(a_j^{h}, b)
- \gamma_j
\bigr| \le 2^{-(h+2)} 
\]
for $0 \le j\le n-1$,
and since $\M_\SS|_{L_{h+1}}$ is the \Fr\ limit of the finite models of
$\TMS\cap \Lww(L_{h+1})$,
we can find a $b_{h+1}\in\SS$  satisfying
\[
\bigl|\dd_\SS(a_j^{h}, b_{h+1})
- \gamma_j
\bigr| \le 2^{-(h-1)} 
\]
for $0 \le j\le n-1$.
We may further assume that
$\dd_\SS(b_h, b_{h+1}) \le 2^{-h},$
as there is a finite metric space 
containing such a $b_{h+1}$
that extends the one
induced by $a_0^h, \ldots,
a_{n-1}^h, b_h$. 
Now, 
for $0 \le j\le n-1$, we have
$\dd_{\SS}(a_j^{h}, a_j^{h+1}) \le 2^{-(h+3)}$, and so
$\dd_{\SS}(a_j^{h}, a_j) \le 2^{-(h+1)};$
hence
\[
\bigl|\dd_\SS(a_j^{h+1}, b_{h+1})
- \gamma_j
\bigr| \le 2^{-(h+2)},
\]
and so $b_{h+1}$ satisfies the inductive claim.
\end{proof}

Although
$T_{U}$ is not itself $\aleph_0$-categorical, 
as shown by the examples $D\UU$, it is approximately
$\aleph_0$-categorical.
Let $\alpha\colon \Nats \to \RatNonneg$ be a bijection,
and 
for each $i\in\Nats$
define the 
finite sublanguage
of $\LMS$ to be
\[L_i \defas \{d_{\alpha(j)} \st 0\le j \le i \}.\]

\begin{proposition}
The theory $T_{U}$ is
approximately
$\aleph_0$-categorical
with witnessing sequence $\<L_i\>_{i\in\Nats}$.
\end{proposition}
\begin{proof}
For every $i\in\Nats$, 
the restriction $T_{U}\cap\Lww(L_i)$
is the theory of the \Fr\ limit of all 
finite models of 
$T_{U}\cap\Lww(L_i)$,
hence $\aleph_0$-categorical.
\end{proof}

\begin{proposition}
\label{UrysohnInvMeas}
The theory $T_{U}$ and witnessing sequence $\<L_i\>_{i\in\Nats}$ satisfy the
assumptions of Theorem~\ref{SplittingTypesTheorem}.
Hence there is an $\sym$-invariant probability measure $m_U$ on
$\Models_{\LMS}$ that is concentrated on the class of models of $T_U$ and that assigns
probability $0$ to each isomorphism class.
\end{proposition}
\begin{proof}
For each $i \in \Nats$, 
the countable model of $T_{U}\cap
\Lww(L_i)$ is isomorphic to $\M_{\Rationals\UU}|_{L_i}$. Its age
has the strong amalgamation property, because the age of
$\M_{\Rationals\UU}$ has the strong amalgamation property.

For each $i\in\Nats$, let $Q_i$ be the
set of quantifier-free $L_i$-types that
are consistent with $T_{U}\cap\Lww(L_i)$.
We will show that $\<Q_i\>_{i\in\Nats}$
has splitting of order $2$.
Let
$j\in\Nats$ and $q\in Q_j$. We 
show that there is
some $j' > j$ such that each quantifier-free $L_j$-type with 
two
free
variables  has a splitting in the language $L_{j'}$.

Let $k$ be the number of free variables of $q$.
There is an
iterated duplicate
$q'$ of $q$ having $2k$-many free variables, and there
is some finite metric space $\SS$ whose positive distances are distinct and
such that 
$q'$
holds of $\M_{\SS}$
(under some ordering of the elements of $\M_{\SS}$).
Let $j'> j$ be such that
\[ \{ \alpha(i) \st 0 \le i \le j' \} \]
partitions $\Rationals$ so that each part contains at 
most
one positive
distance occurring in $\SS$. Let $\qn$ be the quantifier-free $L_{j'}$-type
of $\M_\SS$.
Then $\qn$ is a splitting of $q$ of order $2$.
\end{proof}

As with the Kaleidoscope random graphs above, the measure $m_U$ cannot be
obtained via the methods in \cite{AFP}.
This is because
almost every sample from $m_U$ has nontrivial definable
closure, as we now show.
Let $\N$ be a structure sampled from $m_U$, and consider its corresponding
metric space $\W_\N = (\Nats, \dd_\N)$.
Then with probability $1$, 
for $(i, j), (i', j')\in \Nats^2$
satisfying $i < j$ and $i' < j'$, 
we have 
\[\dd_\N(i, j) \neq \dd_\N(i', j')\]
whenever $(i, j) \neq (i', j')$.

Also $m_U$ does not arise from the standard examples of the
form $D \UU$, as
for any two independent samples $\N_0$, $\N_1$ from $m_U$, the sets of real distances
\[
\{ \dd_{\N_w}(i, j) \st i, j \in\Nats \mathrm{~and~} i \neq j\}
\]
for $w \in \{0, 1\}$ are almost surely disjoint (and so
any two independent samples from $m_U$ are almost surely
non-isomorphic --- as we already knew). As a consequence, a sample $\N$ is
almost surely such that $\W_\N$ is not isometric to $D \UU$
for any countable dense set $D\subseteq \Rplus$.

\section{$G$-orbits admitting $G$-invariant probability measures}
\label{GorbitSec}

In this section we characterize,
for certain
Polish groups 
$G$, those transitive Borel $G$-spaces that admit $G$-invariant measures.
In particular, we do
so
for all countable Polish groups and for countable products of
symmetric groups
on a countable (finite or infinite) set.
Throughout this section, let $(G, \cdot)$ be a Polish group.

\subsection{$S_\infty$-actions}

For a countable first-order language $L$,
recall that
$\Models_L$  is the
space of $L$-structures with underlying set  $\Nats$, with $\logicact\colon \sym
\times \Models_L \rightarrow \Models_L$ the logic action of \sym\ on
$\Models_L$ by permutation of the
underlying set.

Also recall that for any formula $\varphi\in \Lwow(L)$ and any $\ell_1, \ldots,
\ell_n\in\Naturals$, we have defined the collection of models
\[
\llrr{\varphi(\ell_1, \ldots, \ell_n)} \defas
\bigl\{\M \in \Models_L \st \M \models {\varphi}(\ell_1, \ldots,
\ell_n)\bigr\}.
\]

The following is an 
equivalent formulation
of
the main result of \cite{AFP}.

\begin{theorem}[{\cite{AFP}}]\label{AFPcor}
Let $(X, \circ)$ be a transitive Borel \sym-space,
and suppose that
$\iota\colon X \to \Models_L$
is a Borel embedding,
where $L$ is some countable language.
Note that the image of $\iota$ is the \sym-space
\[(\{\M\in \Models_L\st \M \cong \M^*\}, \logicact)\]
consisting of 
the orbit
in $\Models_L$
of some countably infinite $L$-structure $\M^*$ under the action of $\logicact$.
Then $X$ admits an \sym-invariant
probability measure if and only if
$\M^*$ has trivial definable closure.
\end{theorem}

The following well-known result will be useful in our classification
of transitive Borel $S_\infty$-spaces admitting $S_\infty$-invariant probability measures.

\begin{theorem}[{\cite[Theorem~2.7.3]{MR1425877}}]
\label{universality}
Let $L$ be a countable 
language having relation symbols of
arbitrarily high arity.
Then $(\Models_L, \logicact)$ is a universal Borel \sym-space.
\end{theorem}

Note that by Theorem~\ref{universality},
for any transitive Borel \sym-space $(X, \circ)$, we can always find 
an embedding $X\to\Models_L$, where $L$ is as in Theorem~\ref{universality}.
Hence Theorem~\ref{AFPcor} provides a complete characterization of
those transitive Borel \sym-spaces admitting \sym-invariant probability measures.
The main result of this section, Theorem~\ref{maintheorem-new},
is a generalization of Theorem~\ref{AFPcor} to the case of invariance under certain products of symmetric groups.

\subsection{Countable $G$-spaces}
We now 
characterize,
for 
countable
groups 
$G$,
those
transitive Borel
$G$-spaces admitting $G$-invariant probability measures.

\begin{lemma}
\label{finiteLem}
Let $(X, \circ)$ be a finite Borel $G$-space. Then $(X, \circ)$ admits a
$G$-invariant probability measure.
\end{lemma}
\begin{proof}
The counting measure $\rho_X$, given by $\rho_X(A) = |A|/|X|$, is $G$-invariant.
\end{proof}
\begin{corollary}
Suppose $G$ is finite. Then every transitive Borel $G$-space admits an invariant probability measure.
\end{corollary}
\begin{proof}
Because $G$ is finite, every transitive Borel $G$-space is also finite.
By Lemma~\ref{finiteLem}, every such $G$-space admits a $G$-invariant
probability measure.
\end{proof}

\begin{lemma}
\label{transitiveNoInv}
Let $(X, \circ)$ be a countably infinite transitive Borel $G$-space. Then $(X,
\circ)$ does not admit a $G$-invariant probability measure.
\end{lemma}
\begin{proof}
Suppose $\mu_X$ is a $G$-invariant probability measure on $(X, \circ)$.
By the transitivity of $X$, for all $x, y\in X$ we must have
$\mu_X(\{x\}) = \mu_X(\{y\})$.  Let $\alpha\defas \mu_X(\{x\})$.
 As $X$ is countable and $\mu_X$ is countably additive, we have
\[1 = \mu_X(X) = \sum_{x \in X}\mu_X(\{x\})= \sum_{x\in X}\alpha.\]
But this is impossible as $X$ is infinite, and so for any non-zero $\alpha$
the right-hand side is infinite.
\end{proof}

\begin{corollary}
Suppose $G$ is countable. Then a transitive Borel $G$-space $X$ admits a
$G$-invariant probability measure if and only if $X$ is finite.
\end{corollary}
\begin{proof}
As $G$ is countable and $X$ is transitive, $X$ must be countable.
The conclusion then follows from Lemmas~\ref{finiteLem} and \ref{transitiveNoInv}.
\end{proof}

\subsection{Products of symmetric groups}
We now consider those groups $G$ that are a countable product of symmetric groups on countable sets.
For such $G$, we will
characterize those transitive Borel $G$-spaces 
that admit
a $G$-invariant probability measure, 
using
the following 
standard result from descriptive set theory.

Recall the definition of $(\MLZ, \logicact^{\M_0})$ from \S\ref{MLZ-sec}.

\begin{theorem}[{\cite[Theorem~2.7.4]{MR1425877}}]
Let $L$ be a countable language and let $L_0$ be a sublanguage of $L$ such that
$L \setminus L_0$ contains relations of arbitrarily high arity.
Let 
$\M_0 \in \Models_{L_0}$.
Then
$\Aut(\M_0)$ is a closed subgroup of \sym,
and $(\MLZ, \logicact^{\M_0})$ is a universal
$\Aut(\M_0)$-space.
\end{theorem}

Note that
the $\Aut(\M_0)$-orbit of any structure $\M^*\in\MLZ$ is of the form
\[\Orb_{L_0}(\M^*) \defas \bigl\{\M\in \MLZ \st \M \cong \M^*\bigr\}.\]
We will be interested in the case when $L_0$ is a unary language, 
i.e., consists entirely of unary relations.

For completeness, and to fix notation for later, we now recall basic facts about the relationship between universal
$G$-spaces  and structures in a given language, when $G$ is the product of
symmetric groups.
For the remainder of the 
section,
let $\ell_0, \ell_1, \ldots, \ell_\infty$ 
be
finite or countably infinite,
define 
\begin{align*}
\Ginf &\defas S_\infty^{\ell_\infty} \text{~~and}\\
\Gfin &\defas \prod_{n \in\Nats} S_n^{\ell_n},
\end{align*}
and let $G \defas \Ginf \times \Gfin$.

Define the countable language
\[
L_G \defas  \{U_i^\infty\st 1\le i \le \ell_\infty\}\ \cup \
\bigcup_{n\in\Nats}\{U_i^n\st 1\le i \le \ell_n\}\ \cup \
\{\vinf, \vfin\},
\]
consisting of unary relation symbols.
Consider the 
theory
$T_G \subseteq \Lwow(L_G)$
defined by the axioms
\begin{itemize}
\item $(\forall x)
\neg \bigl(U_i^\infty(x) \And U_j^\infty(x)\bigr)$
whenever
$1 \le i < j \le \ell_\infty$,
\item $(\forall x)
\neg \bigl(U_i^n(x) \And U_j^m(x)\bigr)$
for all $n, m\in \Nats$ and $i,j$ such that  $1 \le i \le
\ell_n$ and $1\le j \le \ell_m$
for which $(i, n) \neq (j, m)$,
\item $(\forall x)
\bigl(
\vfin(x) \leftrightarrow
\bigvee_{n \in\Nats}
\bigvee_{1 \le i \le \ell_n}
U_i^n(x)\bigr)$,
\item $(\forall x) \bigl(\vinf(x) \leftrightarrow
\bigvee_{1\le i \le \ell_\infty}U_i^\infty(x)\bigr)$,
\item $(\forall x)\bigl(\vfin(x) \leftrightarrow \neg
\vinf(x)\bigr)$,
\item
for all $i$ such that $1 \le i \le \ell_\infty$, the set
$\{x\st U_i^\infty(x)\}$ is infinite, and
\item for all $n\in \Nats$ and $i$ such that $1 \le i \le \ell_n$, we have
$|\{x\st U_i^n(x)\}| = n$.
\end{itemize}
These axioms are consistent; in particular, they can be realized by any $L_G$-structure partitioned by the
$U$-relations for which
each $U^\infty$ relation is infinite, each $U^n$ relation
has size $n$, the relation $\vinf$ is the union of all
$U^\infty$-relations, and $\vfin$ is the union of all $U^n$ relations.

Fix some $\AG\in \Models_{L_G}$ that is a model of $T_G$.
For each $U$-relation, 
write \linebreak $\widetilde{U}\defas U^\AG =  \{x\in A\st \AG \models U(x)\}$, and similarly for each $V$-relation.
Let $P(\widetilde{U})$ be the collection of permutations of $\widetilde{U}$.

\begin{lemma}
\label{complicatedM}
The group
$G$ is
isomorphic to the automorphism group of $\AG$.
\end{lemma}
\begin{proof}
A permutation of $\Nats$ induces an automorphism of $\AG$ if and only if it
preserves each $U$-relation. Hence $\Aut(\AG)$ is isomorphic to
\[
\textstyle
\prod_{1\le i\le \ell_\infty}
P(\widetilde{U_i^\infty})
\times
\prod_{n \in\Nats}
\prod_{1\le i\le \ell_n}
P(\widetilde{U_i^n}).
\]
However, as each
$P(\widetilde{U_i^\infty})$ is isomorphic to \sym, and each
$P(\widetilde{U_i^n})$ is isomorphic to $S_n$, we have that $\Aut(\AG) \cong
G$.
\end{proof}

\begin{lemma}
\label{unaryLem}
Let $L$ be a countable unary language and $\M$ be a countably infinite $L$-structure.
Then $\Aut(\M)$ is isomorphic to a product of symmetric groups.
\end{lemma}
\begin{proof}
For $x,y\in \M$, define $x\sim y$ to hold when $x$ and $y$ have the same
quantifier-free $L$-type.
Let $E$ be the collection of $\sim$-equivalence classes.
As $L$ is unary,
the automorphisms of $\M$ are precisely those
permutations of the underlying set of $\M$ that preserve $\sim$.
Hence $\Aut(\M)\cong \prod_{Y \in E} S_{|Y|}$.
\end{proof}

Note that
Lemmas \ref{complicatedM} and  \ref{unaryLem}
imply the standard fact
that the countable products of
symmetric groups on countable (finite or infinite) sets are precisely those groups isomorphic to automorphisms of
structures in countable unary languages.

\subsection{Non-existence of invariant probability measures}

Recall that
$G = \Aut(\AG)$ by
Lemma~\ref{complicatedM}.
For the rest of the section, fix
a countable relational language $L$ that extends $L_G$.

We now classify those orbits in $\MLG$ that admit
an $\Aut(\AG)$-invariant probability measure. 
Then in
particular, if $L \setminus L_G$ has
relations of arbitrarily high arity, then  $\MLG$ will
be a universal $G$-space, and so we will obtain a classification
of those transitive $G$-spaces 
that admit
$G$-invariant probability measures.

Notice that 
in any
structure $\M \in \MLG$, the 
algebraic
closure of the empty set
contains $\vfin^\AG$, which is non-empty precisely when $G$ is not a
countable power of $S_\infty$. 
Hence, when $\vfin^\AG$ is non-empty, $\M$ does not have trivial definable closure.
To deal with this issue, we define the following notion.

\begin{definition}
An $L$-structure $\M\in \MLG$ has \defn{almost-trivial definable
closure} if and only if for every tuple $\aa \in \M$, we have
\[\dcl(\aa \cup \vfin^\AG) = \aa \cup \vfin^\AG .\]
\end{definition}
Note that the analogous notion of almost-trivial algebraic closure
coincides with almost-trivial definable closure, similarly to 
the way that
trivial
definable closure and trivial algebraic closure coincide.
Using this notion, we can now state our main classification.

\begin{theorem}
\label{maintheorem-new}
Let $\M\in \MLG$.
Then $\Orb_{L_G}(\M)$ 
admits
a $G$-invariant probability measure if and only
if $\cM$ has almost-trivial definable closure.
\end{theorem}

We will prove Theorem~\ref{maintheorem-new}
in two steps. We 
prove
the forward direction in
Proposition~\ref{forward}. This argument is very similar to an analogous result
in \cite{AFP}, 
but
we include it here for completeness. 
In
Proposition~\ref{backward}, we prove the reverse direction.

\begin{proposition}
\label{forward}
Let $\M\in \MLG$, and suppose
that $\Orb_{L_G}(\M)$ 
admits
a \linebreak $G$-invariant probability measure.
Then $\cM$ has almost-trivial definable closure.
\end{proposition}
\begin{proof}
Let $\mu$ be a $G$-invariant probability measure on
$\Orb_{L_G}(\M)$, and
suppose that there 
is
a finite tuple $\aa \in \cM$ such that
\[b \in \dcl(\aa \cup \vfin^{\AG}) \setminus (\aa \cup \vfin^{\AG}). \]

Let 
$p(\x y)$ be a formula that generates a (principal)
$\Lwow(L)$-type of $\aa b$, i.e., a formula of $\Lwow(L)$ 
with free variables $\x y$
such that for any
$\Lwow(L)$-formula
$\psi$
whose free variables are among  $\x y$,
either
\[\models (\forall \x)(\forall y) \bigl(
p(\x y) \to \psi (\x y)\bigr)
\qquad \text{or}
\qquad
\models (\forall \x)(\forall y) \bigl( p (\x y)
\to \neg \psi (\x y) \bigr).
\]

Because $\M \models (\exists \x y)\, p(\x y)$, the measure $\mu$ is concentrated on
$\llrrAG{(\exists \x y)\, p(\x y)}$.
By the countable additivity of $\mu$, there is some $\m \in \Nats$ such that 
$\mu\bigl(\llrrAG{(\exists y)\, p(\m y)} \bigr) > 0$.

Now, $b\not \in \vfin^{\AG}$,
and so $b \in \vinf^{\AG}$.
Hence we must have $\cM\models U_k^\infty(b)$ for
some $k$ such that $1 \le k \le \ell_\infty$.
Let
\[
F \defas \{n^*\in\Nats\st \AG \models U_k^\infty(n^*) \text{~and~} n^*\not\in\m\}.
\]

As $b \not \in \aa$,  note that 
$
\llrrAG{(\exists y)\, p(\m y)}
=
\bigcup_{n\in F} \llrrAG{p(\m n)}
$.
Because 
$
b \in \dcl(\aa \cup \vfin^{\AG}) \setminus (\aa \cup \vfin^{\AG}), 
$
for any distinct $n_0, n_1 \in F$ we have
$
\llrrAG{p(\m n_0)}
\cap
\llrrAG{p(\m n_1)} 
= \emptyset,
$
and so
$
\mu \bigl (\llrrAG{(\exists y)\, p(\m y)} \bigr )
=
\sum_{n^*\in F} \mu \bigl (\llrrAG{p(\m n^*)} \bigr ).
$

By countable additivity, there is some $n\in F$ such that 
$
\alpha \defas
\mu\bigl(\llrrAG{p(\m n)}\bigr) > 0.
$
Further, by the definition of $F$, for every $n^*\in F$ there is some $g \in G$ such that 
\linebreak
$
g(\m n) = \m n^*
$
and $g$ fixes $\vfin^\AG$.
As $\mu$ is $G$-invariant, for all $n^*\in F$ we have
\linebreak
$
\mu\bigl(\llrrAG{p(\m n^*)}\bigr) 
=
\mu\bigl(\llrrAG{p(\m n)}\bigr),
$
and so 
$
\mu \bigl (\llrrAG{(\exists y)\, p(\m y)} \bigr )
=
\sum_{n^*\in F}\alpha.
$
This is a contradiction, as 
$\alpha> 0$
and
$F$ is 
infinite.
\end{proof}

This concludes the forward direction of
Theorem~\ref{maintheorem-new}.

\subsection{Constructing the invariant probability measure}

The reverse direction of
Theorem~\ref{maintheorem-new} will use the construction in Section~\ref{invlimitconstruction}
analogously to the way in which
the main construction in
\cite{AFP} 
is used
to classify those
transitive $\sym$-spaces admitting
$\sym$-invariant probability measures.

\begin{lemma}
\label{bijectionLemma}
Let $\M \in \MLG$, and suppose that
$\mu$ is a $\Ginf$-invariant probability measure on
$\Orb_{L_G}(\M)$.
Then there is a $G$-invariant probability measure
$\mufin$
on $\Orb_{L_G}(\M)$.
\end{lemma}
\begin{proof}
First note that, for each
$n\in \Nats$ and $1\le i\le \ell_n$,
there is a unique order-preserving bijection
\[\iota_i^n\colon \widetilde{U_i^n} \to \{1, \ldots, n\}.\]
Recall that 
these
relations $\widetilde{U_i^n}$, along with 
$\widetilde{U_i^\infty}$,
partition $\AG$.
Define the maps
\begin{align*}
\alpha&\colon \Nats \to \Nats \text{~~and}\\
\beta&\colon \Nats \to \Nats \cup \{\infty\}
\end{align*}
to
be such that for all $n\in\Nats$,
\[ \AG \models U_{\alpha(n)}^{\beta(n)}(n)
.
\]
For every finite subset $Y\subseteq \Nats$, let
\[Y^*\defas \bigcup_{y\in   Y}  \widetilde{U_{\alpha(y)}^{\beta(y)}}.
\]
Further, define the finite group
\[G_Y \defas \prod_{a,b \in\Nats} \bigl\{ S_b
\st
(\exists y \in Y)\, (\alpha(y) = a \text{~and~} \beta(y) =b )
\bigr
\}.
\]
In other words, $G_Y$ 
contains the product of
$\bigl|\{\alpha(y)\st y \in Y \text{~and~} \beta(y) = b\} \bigr|$-many
copies of $S_{b}$.

There is a natural action of $G_Y$ on $Y^*$ that
fixes $\widetilde{\vinf}$ pointwise, and uses the $\alpha(y)$-th copy of
$S_{\beta(y)}$ to permute $\widetilde{U_{\alpha(y)}^{\beta(y)}}$.

We will define $\mufin$ via a
sampling procedure.
Begin by sampling an element \linebreak $\N^* \in \Orb_{L_G}(\M)$
according to $\mu$.
Next, for each
unary relation  $U_i^n$ where $n\in \Nats$ and $1\le i\le \ell_n$,
independently select an element $\sigma_i^n$ of $S_n$, uniformly at random.
Finally, let 
$\mufin$
be the distribution of the
structure $\N\in\MLG$ defined as follows.
For every relation symbol $R\in L$ and every $h_1, \ldots, h_j\in\Nats$, where $j$
is the arity of $R$, let
\[
\N \models R(h_1, \ldots, h_j)
\quad \text{iff} \quad
\N^* \models R(h_1^*, \ldots, h_j^*),
\]
where for $1 \le p \le j$,
when $h_p^*\in \widetilde{U_i^n}$ for some $n\in\Nats$ and $i$ such that $1\le i\le \ell_n$, we have
\[\bigl((\iota_i^n)^{-1}\sigma_i^n\iota_i^n\bigr)\bigl(h_p^*\bigr) = h_p,
\]
and when $h_p^*\in \widetilde{\vinf}$, we have $h_p^* = h_p$. Now, $\N$ is
almost surely isomorphic to $\N^*$ via the isomorphism that is the identity
on $\widetilde{\vinf}$ and is $(\iota_i^n)^{-1}\sigma_i^n\iota_i^n$ on each
$\widetilde{U_i^n}$.
Thus
$\mufin$ is is a measure on
$\Orb_{L_G}(\M)$, as claimed.

We now show that the probability measure $\mufin$ is $\Gfin$-invariant.
Because, in the definition of $\mufin$, each finite permutation
$\sigma_i^n$ was selected uniformly independently from $S_n$, we have
\[\mufin\bigl(\llrrAG{R(h_1, \ldots, h_j)}\bigr) =
\frac1{\bigl| G_{\{h_1, \ldots, h_j\}}\bigr|}
\,
{\sum_{g\in G_{\{h_1, \ldots, h_j\}}} \mu\Bigl(
\bigllrrAG{R\bigl(g(h_1), \ldots, g(h_j)\bigr)}
\Bigr)}
,
\]
where each $g \in G_{ \{ h_1, \ldots, h_j \} } $
acts on each $h_p$ (for $1\le p\le j$) as described above.

Note, however, that  for all $g^* \in \Gfin$, there is some $g \in
G_{\{h_1, \ldots, h_j\}}$ such that  the actions of $g$ and $g^*$  agree on
$\{h_1, \ldots, h_j\}$.
Hence
\[\mufin\Bigl(\bigllrrAG{R\bigl(g^*(h_1), \ldots, g^*(h_j)\bigr)}\Bigr)
=
\mufin\bigl(\llrrAG{R(h_1, \ldots, h_j)}\bigr),
\]
and so $\mufin$ is $\Gfin$-invariant.

Recall that $\mu$ is $\Ginf$-invariant. We now show that $\mufin$ is also
$\Ginf$-invariant, so that
$\mufin$ is invariant under $G = \Ginf \times \Gfin$, as desired.
Let $f\in \Ginf$, let $R\in L$ be a relation symbol, and let $j$ be the
arity of $R$. We now show that, for all $h_1, \ldots, h_j \in \Nats$,
\begin{eqnarray*}
&&
\hspace{-36pt}
\mufin\Bigl(\bigllrrAG{R\bigl(f(h_1), \ldots, f(h_j)\bigr)}\Bigr)
\\
&=& \frac1{\bigl| G_{\{f(h_1), \ldots, f(h_j)\}}\bigr|}
\,
{\sum_{g\in G_{\{f(h_1), \ldots, f(h_j)\}}} \mu\Bigl(
\bigllrrAG{R\bigl(g(f(h_1)), \ldots, g(f(h_j))\bigr)}
\Bigr)}\\
&=& \frac1{\bigl| G_{\{h_1, \ldots, h_j\}}\bigr|}
\,
{\sum_{g\in G_{\{h_1, \ldots, h_j\}}} \mu\Bigl(
\bigllrrAG{R\bigl(g(h_1), \ldots, g(h_j)\bigr)}
\Bigr)}\\
&=& \mufin\bigl(\llrrAG{R(h_1, \ldots, h_j)} \bigr)
,
\end{eqnarray*}
where each $g \in G_{ \{ h_1, \ldots, h_j \} } $
again acts on each $g(h_p)$ and $h_p$ (for $1\le p\le j$) as described above.
The first and third equalities are as before.
Note that  $f$ is the identity on  $\widetilde{\vfin}$ and so
$G_{\{f(h_1), \ldots, f(h_j)\}} = G_{\{h_1, \ldots, h_j\}}$;
the second equality
follows from this and our assumption that
$\mu$ is $\Ginf$-invariant.
Therefore $\mufin$ is $\Ginf$-invariant, hence $G$-invariant.
\end{proof}

\begin{proposition}
\label{backward}
Let $\M\in \MLG$, and suppose
that $\cM$ has almost-trivial definable closure.
Then $\Orb_{L_G}(\M)$ has a $G$-invariant probability measure.
\end{proposition}
\begin{proof}
There are two cases. Suppose $\widetilde{\vinf}$ is empty. In this case, $\Ginf$ is
the trivial group, and so every measure on $\Orb_{L_G}(\M)$ is $\Ginf$-invariant.

Otherwise, $\widetilde{\vinf}$ is non-empty.
Hence $U_1^\infty \in L_G$, and
so $\widetilde{U_1^\infty}$ is a countably infinite set.
Therefore $\widetilde{\vinf}$ is countably infinite,
and so there is a
bijection $\tau \colon \widetilde{\vinf} \to \Nats$. Let $\M^\tau \in
\Models_{\cC_0, L}$ be such that for any quantifier-free $L$-type $q$,
\[
\M^\tau \models q(h_1, \ldots, h_j)
\quad \text{iff} \quad
\M \models q\bigl(\tau^{-1}(h_1), \ldots, \tau^{-1}(h_j)\bigr),
\]
where $j$ is the number of free variables of $q$.

Fix some countable admissible set $A$ containing
the Scott sentence $\sigma$ of $\M^\tau$ (equivalently, of $\M$).
Let the $L_A$-theory $\fT_A$
be the definitional expansion
(as in Lemma~\ref{morleyization}) of 
$A$.
Let $T_A \defas \fT_A \cup \{\sigma_A\}$, where $\sigma_A \in L_A$ is a pithy $\Pi_2$ sentence such that
\[
\fT_A \models \sigma_A \leftrightarrow \sigma
.
\]
For each $i\in\Nats$ define the language $L_i \defas  L_A$ and theory 
$T_i \defas T_A$, and
 let $Q_i$ be any enumeration of all quantifier-free $L_A$-types
over $A$ (of which there are only countably many).

Let $\M^\tau_A$ be the unique expansion of $\M^\tau$ to a model of $\fT_A$.
We will now show that there is an $\sym^{C_0}$-invariant probability measure
on $\Models_{\cC_0, L_A}$ that is concentrated on the class of models of $T_A$.
We will do so by showing that $\<Q_i\>_{i\in\Nats}$ satisfies conditions (W), (D), (E) and
(C) of our main construction, and so Proposition~\ref{asModel} applies.
Now, (W), (E), and (C) follow immediately as each $Q_i$
enumerates
all quantifier-free types consistent with $T_i = T_A$.

Suppose we do not have condition (D), i.e., duplication of quantifier-free
types.  Then  there is some $i\in\Nats$, some non-redundant non-constant quantifier-free type $q\in
Q_i$, and some tuple $\a \in \M^\tau_A$ such that there is a unique $b \in
\vinf^{\M^\tau_A}$ (as $q$ is non-constant) for which
\[
\M^\tau_A \models
q(\a, b).
\]
In particular, if $g\in \Aut(\M^\tau_A)$ fixes $\a \cup \vinf^{\M^\tau_A}$
pointwise, then  $g(b) = b$, and so
$\M^\tau_A$ does not have
almost-trivial definable closure (since
$b$ is disjoint from $\a$ as $q$ is non-redundant).
This violates our assumption of
almost-trivial definable closure 
for
$\M$, as 
$\M$
is isomorphic to
$\M^\tau_A$.
Hence condition (D) holds, and so by Proposition~\ref{asModel}
there is an invariant measure $m_\infty^\circ$ on $\Models_{\cC_0, L_A}$ that is
concentrated on the class of models of $T_A$, i.e., the isomorphism class of $\M^\tau_A$.

Now let $\mu$ be the probability measure on $\MLG$ satisfying, for
any relation symbol $R\in L$,
\[
\mu\bigl(\llrrAG{R(h_1, \ldots, h_j)}\bigr)
=
\mu_\infty^\circ \bigl(
\llrrC{R(\tau(h_1), \ldots, \tau(h_j)}
\bigr)
,
\]
where $j$ is the arity of $R$. 
The measure
$\mu$  is concentrated on
$\Orb_{L_G}(\M)$, as $m_\infty^\circ$ is concentrated on the isomorphism
class of $\M^\tau_A$. Hence the restriction $\mu'$ of $\mu$ to
$\Orb_{L_G}(\M)$ is a probability measure.
Furthermore, $\mu$ is
$\Ginf$-invariant  because $m_\infty^\circ$ is $\sym^{C_0}$-invariant.

By Lemma~\ref{bijectionLemma} applied to $\M$ and $\mu'$, there is a
$G$-invariant probability measure  on $\Orb_{L_G}(\M)$.
\end{proof}

This concludes the reverse direction of Theorem~\ref{maintheorem-new}.

\section{Concluding remarks}

In this paper we have provided conditions under which
the class of models of 
a theory admits an
invariant 
measure that is
not concentrated on any single isomorphism class.
But much remains to be explored.
In particular, there are 
natural constructions of invariant measures
that do not arise by the
techniques that we have described, but which would be interesting to capture 
through
general constructions.

\subsection{Other invariant measures}

The best-known invariant measures concentrated on the Rado graph
are
the distributions of the
countably infinite \ER\ random graphs
$\GG(\Nats, p)$ for $0 < p <
1$, 
in which
edges are chosen independently 
using
weight $p$ coins. These are not 
produced by our
constructions. 
In particular, when considered as arising from dense graph
limits, 
these limits
all have positive
entropy (as defined in, e.g., \cite[\S D.2]{MR3043217}), while any
of our invariant measures concentrated on graphs 
corresponds 
to a dense graph
limit that has zero entropy; equivalently, our measures arise from graphons
that are $\{0,1\}$-valued a.e., or
``random-free'' (see \cite[\S10]{MR3043217}).

\subsubsection{Kaleidoscope theories}
A similar phenomenon occurs with the following natural construction of an
invariant measure concentrated on the class of models of the Kaleidoscope 
theory built from
certain
ages.
Consider an age $A$ in a language $L$, both satisfying the hypotheses of
Proposition~\ref{KaleidoscopeInvMeas},  and let $n\in\Nats$
be such that 
$A$ has at least two
non-equal elements of size $n$ on the same underlying set.

Since $A$ is a strong amalgamation class,
there is 
some
invariant measure $\mu$ concentrated on the (isomorphism class
of the) \Fr\ limit of $A$, as proved in \cite{AFP}.
We now describe
an invariant measure,
constructed using $\mu$, that is
concentrated on the 
class of models of the Kaleidoscope 
theory $T_\infty$ built from the age
$A$.

Namely, consider the distribution $\mu_\infty$ of the following random construction.
Let $\X$ be a random structure in $\Models_{L_\infty}$ such that for each $i\in\Nats$, 
$\X|_{L^i}$ is
an $L^i$-structure consisting 
of
an
independent 
sample
from $\mu$.   Observe that this procedure almost surely
produces a  model of 
$T_\infty$, and so $\mu_\infty$ is
an invariant measure concentrated on the class of models of 
$T_\infty$.

For any $n$-tuple $\a\in \Nats$ and any distinct $i, j\in\Nats$, the 
random quantifier-free
$L^i$-type of  $\a$ induced by sampling from $\mu_\infty$ is independent
from the random quantifier-free \linebreak $L^j$-type of $\a$.
Hence the set of structures realizing any given quantifier-free $L_\infty$-type in
$n$ variables has measure $0$,
and so $\mu_\infty$ assigns measure $0$ to any single isomorphism class.
Furthermore, for ages consisting of graphs, when $\mu$ is not random-free,
one can show that the resulting invariant measure is not captured via our
constructions above.

For example, consider the case of the Kaleidoscope random graphs, where $\mu$ is
the distribution of the \ER\ graph $\GG(\Nats, 1/2)$, in which
edges are determined
by independent flips of a fair coin.
Then $\mu_\infty$ is
an invariant measure determined by independently flipping a fair coin to determine the
presence of a $c$-colored edge for each pair of vertices, for each of countably
many colors $c$.
The measure $\mu_\infty$ is concentrated
on the class of Kaleidoscope random graphs and
assigns measure $0$ to 
each 
isomorphism class, but does not arise via our methods.

\subsubsection{Urysohn space}

Likewise, there is another  natural invariant measure on $\Models_{\LMS}$ concentrated on
the class of countable $\LMS$-structures $\N$ that are models of $T_U$ (i.e., such that 
the completion of $\W_\N$ is $\UU$),
but which assigns measure $0$ to each isomorphism class.

Namely, for any
countable dense set $D \subseteq \Rplus$, recall that $D\UU$ is
the metric space induced by the \Fr\ limit of  all finite metric spaces (considered as
$\LMS$-structures) whose set of non-zero distances 
is contained in $D$.
Note that 
for any such $D$, the \linebreak $\LMS$-structure $\M_{D\UU}$ has trivial
definable closure (unlike the $\LMS$-structure corresponding to a typical sample of
the invariant measure $m_U$ that we constructed in
Proposition~\ref{UrysohnInvMeas}).
Hence, 
as proved in \cite{AFP},
there is an invariant
measure $m_D$ on $\Models_{\LMS}$, concentrated on the isomorphism class
of $\M_{D\UU}$.

Now let $\widetilde{D}$ be a random subset of $\Rplus$
chosen via 
a countably infinite set of independent samples from any non-degenerate 
atomless
probability measure on
$\Rplus$. 
Then
with probability $1$, the set $\widetilde{D}$ is
infinite, dense, and for any given $r\in \Rplus$ does not contain $r$.
Finally, consider the random measure $m_{\widetilde{D}}$. Its distribution
is also an  invariant measure on $\Models_{\LMS}$ concentrated on
the class of countable $\LMS$-structures $\N$ such that
the completion of the corresponding metric space $\W_\N$
is isometric to $\UU$,
but which assigns measure $0$ to each isomorphism class.
However, this invariant measure is different from the measure $m_U$ that we
constructed
in Proposition~\ref{UrysohnInvMeas}, as a typical sample from it has
trivial definable closure, whereas a typical sample from $m_U$ does not.

We now discuss a more elaborate case of invariant measures that can
also be described explicitly but which do not arise from our
construction.
This set of examples, along with 
the
explicit Kaleidoscope and Urysohn
constructions described above,
motivate the search for further general conditions that lead to invariant
measures.

\subsubsection{Continuous transformations}
\label{cont-subsubsec}

The previous example involved no relationship between the various 
copies $L^j$ of the original language.  We now consider a more complex example,
in which interactions 
within
a sequence of languages allow us to describe
``transformations'' from one structure to another.
Although the invariant measure in this example will assign measure $0$ to every isomorphism class, it is not clear how it could arise from the methods of this paper.

Let $L$ be a countable relational language.
Consider the larger language $\Ltr$, which consists of the disjoint
union of countably
infinitely many copies $L^t$ of
$L$ indexed by $t\in \Rationals \cap [0,1]$.
For each relation
symbol $R\in L$,
write $R^t$ for the corresponding symbol indexed by $t\in\Rationals\cap
[0,1]$.
One can think of the $\Ltr$-structure as describing 
a ``time-evolution''
starting with 
a 
structure which occurs in the first sublanguage $L^0$, 
and ending at
another structure
which occurs in the last sublanguage $L^1$, 
progressing through structures in intermediate sublanguages.

\begin{definition}
Let $\M_0$ be an $L^0$-structure and $\M_1$ an $L^1$-structure.
We call  an $\Ltr$-structure $\M$ a \defn{transformation} of
$\M_0$ into $\M_1$
when
\[\M|_{L^0}  = \M_0 \quad \text{and} \quad  \M|_{L^1}  = \M_1,\]
and for all relation symbols $R\in L$,
where $n$ is the arity of $R$, and all $s, t \in \Rationals$ such that
$0 \le s < t \le 1$,
\[\M \models (\forall x_1, \ldots, x_n) \bigl (R^s(x_1, \ldots, x_n)
\rightarrow R^t(x_1, \ldots, x_n)
\bigr).
\]
\end{definition}

We now define a notion,
called a \emph{nesting},
that will ensure coherence between structures in languages with
intermediate indices, as ``time'' progresses.

\begin{definition}
\label{nestingdef}
Suppose $A_0$ is an age in the language $L^0$ and $A_1$ is an age in the language $L^1$.
We define
a \defn{nesting of $A_0$
in $A_1$} to be an age $A$ in the language $L^0 \cup L^1$
that satisfies the following properties:
\begin{itemize}
\item $A$ 
is a strong amalgamation class.
\item For every $\cK\in A$ and every relation $R$ in $L$,
\[\cK \models (\forall x_1, \ldots, x_n) \bigl (R^0(x_1, \ldots x_n) \rightarrow R^1(x_1, \ldots x_n)
\bigr),
\]
where $n$ is the arity of $R$.
\item If $\cN$ is a \Fr\ limit of $A$, then $\cN|_{L^0}$ is a \Fr\ limit of
$A_0$ and $\cN|_{L^1}$ is a \Fr\ limit of
$A_1$.
\end{itemize}
\end{definition}

For example, consider the age consisting of 
all those
ways that a finite graph can be
overlaid on a finite triangle-free graph (using a different edge relation)
such that whenever there is an edge in the latter there is a corresponding
edge in the former.  This is a nesting of the collection of finite
triangle-free graphs in the collection of finite graphs. The \Fr\ limit of
the joint age consists of a copy of the Rado graph overlaid on a copy of the
Henson triangle-free graph (using different edge relations) such that
whenever a pair of vertices has an edge in the latter, it has one in the
former.

Given
a nesting $A$ of $A_0$ in $A_1$ as in Definition~\ref{nestingdef},
we will
now describe a \emph{random} \linebreak $\Ltr$-structure $\M$ that is a.s.\ a transformation of $\M|_{L^0}$ into $\M|_{L^1}$, and for which
$\M|_{L^0 \cup L^1}$ is a \Fr\ limit of $A$, almost surely.
Furthermore, the distribution of $\M$ will be invariant under arbitrary
permutations of the underlying set.

Because $A$ has the strong amalgamation property,
 there is some probability measure $\mu$ on
$\Models_{L^0 \cup L^1}$,
invariant under
$\sym$,
that is concentrated on the isomorphism class of the \Fr\ limit of $A$.
Our procedure starts by first  sampling $\mu$ to obtain a random
structure $\cN \in \Models_{L^0 \cup L^1}$.

Conditioned on $\cN$, for every relation symbol $R \in L$ and
every $j_1, \ldots, j_n \in \Nats$, where $n$ is the arity of $R$,
choose $r_{R, j_1, \ldots, j_n}\in\Reals$
as follows.  
If
\[
\cN\models \neg R^0(j_1, \ldots, j_n) \And R^1(j_1, \ldots, j_n),
\]
then independently choose a real number $r_{R, j_1, \ldots, j_n}\in(0,1)$ uniformly at
random;  if
\[
\cN\models \neg R^0(j_1, \ldots, j_n) \And \neg R^1(j_1, \ldots, j_n),
\]
then
let $r_{R, j_1, \ldots, j_n} \defas 2$, so that $R^s(j_1, \ldots, j_n)$ will 
not hold for any $s$;
otherwise let
$r_{R, j_1, \ldots, j_n} \defas 0$.
Define $\M$ to be the $\Ltr$-structure such that
for all $s \in \Rationals \cap [0,1]$,
\[\M \models R^s(j_1, \ldots, j_n)
\]
if and only if $s \ge r_{R, j_1, \ldots, j_n}$,
for all $R \in L$ and
every $j_1, \ldots, j_n \in \Nats$, where $n$ is the arity of $R$.

The real $r_{R, j_1, \ldots, j_n}$ can be thought of as the
point in time at which $R(j_1, \ldots, j_n)$ 
``appears'', in that it flips
from not holding (in sublanguages
$L^s$ for $s < r_{R, j_1, \ldots, j_n}$) to holding (in sublanguages $L^s$
for $s \ge r_{R, j_1, \ldots, j_n}$).
Each $\M|_{L^s}$ 
then provides
a ``snapshot'' of
the structure over time as it
transitions from $\M|_{L^0}$ to $\M|_{L^1}$, 
whereby the relations hold 
of
more and more tuples.
In particular, for any tuple
and relation (of the same arity),
the set of ``times'' for which the relation holds of the tuple
is upwards-closed.

Note that
whenever there are 
such
points 
$r_{R, j_1, \ldots, j_n}$
other than $0$ and $2$, i.e., when
there is some tuple 
of
which a relation holds in $\M|_{L^1}$  but not
in $\M|_{L^0}$,
then
any two independent samples from the distribution of
$\M$ are a.s.\ non-isomorphic, as their respective sets of transition points
are a.s.\ distinct.
Hence, under this hypothesis, the distribution of $\M$ is an invariant
measure that assigns measure $0$ to every isomorphism class of
$\Ltr$-structures.

\subsection{Open questions}

In this paper, we have given conditions 
on a first-order
theory 
that ensure the existence of
an invariant measure concentrated on the class of its models
but on no single isomorphism class;
but a complete characterization has yet to be determined. It
would be interesting also
to characterize the
structure of these invariant measures.

Another 
question
is to find conditions under which
one can formulate
similar 
results for appropriate models of more sparse
structures. 
Various notions of sparse graphs and intermediate classes have recently
been studied extensively (see, e.g., \cite{MR2920058} and \cite{NO-unified});
for a presentation of graph limits for bounded-degree graphs, see \cite{MR3012035}. 

One may also ask whether one can obtain measures concentrated on the 
class of models of the 
theory of continuous transformations
described 
in \S\ref{cont-subsubsec}, and still not on any single isomorphism class,
in
a ``random-free'' way, i.e.,
by sampling from a
(two-valued) continuum-sized structure,
as in our main construction.

\section*{Acknowledgments}

This research was facilitated by participation in the workshop on
Graph and Hypergraph Limits at the American Institute of Mathematics 
(Palo Alto, CA),
the second workshop on Graph Limits, Homomorphisms and Structures at
Hrani\v{c}n\'i Z\'ame\v{c}ek 
(Czech Republic), the conference on Graphs and Analysis  at the
Institute for \linebreak Advanced Study 
(Princeton, NJ), 
the Arbeitsgemeinschaft on Limits of Structures at the Mathematisches
Forschungsinstitut Oberwolfach 
(Germany),
and the Workshop on Homogeneous Structures of the Hausdorff Trimester
Program on Universality and Homogeneity at the Hausdorff Research Institute
for Mathematics (Bonn, Germany).

Work on
this publication by CF was made possible through the support of
NSF grant DMS-0901020,
ARO grant W911NF-13-1-0212,
and grants from the
John Templeton Foundation and Google.
The opinions expressed in this
publication are those of the authors and do not necessarily reflect the
views of the John
Templeton Foundation or the U.S.\ Government.

Work by JN has been partially supported by the 
Project LL-1201 ERCCZ  CORES and by CE-ITI P202/12/G061 of the GA\v{C}R.

The authors would like to thank Daniel Roy and the referee for helpful comments, Alex Kruckman for detailed suggestions on an earlier draft,
and 
Robert Israel for the statement and proof of
Lemma~\ref{veryCombLemma}.


\begin{small}
\bibliographystyle{amsnomr}

\def\cprime{$'$} \def\polhk#1{\setbox0=\hbox{#1}{\ooalign{\hidewidth
  \lower1.5ex\hbox{`}\hidewidth\crcr\unhbox0}}}
  \def\polhk#1{\setbox0=\hbox{#1}{\ooalign{\hidewidth
  \lower1.5ex\hbox{`}\hidewidth\crcr\unhbox0}}} \def\cprime{$'$}
  \def\cprime{$'$} \def\cprime{$'$} \def\cprime{$'$} \def\cprime{$'$}
  \def\cprime{$'$} \def\cprime{$'$} \def\cprime{$'$} \def\cprime{$'$}
\providecommand{\bysame}{\leavevmode\hbox to3em{\hrulefill}\thinspace}
\providecommand{\MR}{\relax\ifhmode\unskip\space\fi MR }
\providecommand{\MRhref}[2]{%
  \href{http://www.ams.org/mathscinet-getitem?mr=#1}{#2}
}
\providecommand{\href}[2]{#2}

\end{small}

\end{document}